\documentclass[a4paper,10pt]{article}
\usepackage[utf8]{inputenc}
\usepackage[english]{babel}

\usepackage{amsmath, amsfonts, amsthm, amssymb, amstext}
\usepackage{dsfont}
\usepackage{bbm,bm}
\usepackage{enumitem}

\usepackage{graphicx}
\usepackage[table]{xcolor}
\usepackage{url}

\usepackage{tocloft}
\usepackage{hyperref}
\usepackage{bookmark}
\hypersetup{
   linktoc=page,
   colorlinks=true,
   linkcolor=blue,
   citecolor=blue,
   urlcolor=blue,
}

\usepackage[authoryear,round]{natbib}

\theoremstyle{plain}
\newtheorem{lemma}{Lemma}
\newtheorem{theorem}{Theorem}

\newtheorem{prop}{Proposition}

\theoremstyle{definition}
\newtheorem{example}{Example}

\newcommand{\cA}{\mathcal{A}}

\newcommand{\cL}{\mathcal{L}}
\newcommand{\cF}{\mathcal{F}}
\newcommand{\cP}{\mathcal{P}}
\newcommand{\cB}{\mathcal{B}}
\newcommand{\cX}{\mathcal{X}}
\newcommand{\cC}{\mathcal{C}}
\newcommand{\cH}{\mathcal{H}}

\newcommand{\cT}{\mathcal{T}}
\newcommand{\sA}{\mathsf{A}}
\newcommand{\sO}{\mathsf{O}}
\newcommand{\real}{\mathbb{R}}
\newcommand{\RR}{\mathbb{R}}

\newcommand{\dd}{\, \mathrm{d}}
\newcommand{\E}{\mathbb{E}}

\newcommand{\Mz}{\mathbb{M}_{0}}
\newcommand{\Mf}{\mathcal{M}_{\rm f}}
\newcommand{\one}[1]{\ensuremath{\mathbbm{1}}(#1)}
\newcommand{\hsp}{\hspace{0.25mm}}

\DeclareMathOperator{\diam}{diam}

\renewcommand{\natural}{\mathbb{N}}
\newcommand{\nat}{\mathbb{N}}
\renewcommand{\P}{\mathbb{P}}

\newcommand{\suppIntensities}{S1}

\newcommand{\DMalpha}{0.05}
\newcommand{\Nsim}{100}
\newcommand{\Msim}{500}

\allowdisplaybreaks

\begin{document}

\title{Comparative evaluation of point process forecasts}

\author{Jonas R. Brehmer\thanks{Computational Statistics group, Heidelberg Institute for Theoretical Studies, Germany; e-mail:   \texttt{jonas.brehmer@h-its.org} }    \and
        Tilmann Gneiting\thanks{Computational Statistics  group, Heidelberg Institute for Theoretical Studies, and Institute for Stochastics, Karlsruhe Institute of Technology, Germany; e-mail: \texttt{tilmann.gneiting@h-its.org}}   \and
        Marcus Herrmann\thanks{Department of Earth, Environmental, and Resources Sciences, University of Naples Federico II, Italy; e-mail:
           \texttt{marcus.herrmann@unina.it}}    \and
        Warner Marzocchi\thanks{Department of Earth, Environmental, and Resources Sciences, University of Naples Federico II, Italy; e-mail:
           \texttt{warner.marzocchi@unina.it}}   \and
        Martin Schlather\thanks{Institute of Mathematics, University of Mannheim, Germany; e-mail:
           \texttt{schlather@math.uni-mannheim.de}}   \and
        Kirstin Strokorb\thanks{School of Mathematics, Cardiff University, Wales, United Kingdom; e-mail:
           \texttt{strokorbk@cardiff.ac.uk}
        }
}
\date{July 9, 2023}

\maketitle

\begin{abstract}

Stochastic models of point patterns in space and time are widely
used to issue forecasts or assess risk, and often they affect
societally relevant decisions.  We adapt the concept of consistent
scoring functions and proper scoring rules, which are statistically
principled tools for the comparative evaluation of predictive
performance, to the point process setting, and place both new and
existing methodology in this framework.  With reference to 
earthquake likelihood model testing, we demonstrate that
extant techniques apply in much broader contexts than previously
thought.  In particular, the Poisson log-likelihood can be used for
theoretically principled comparative forecast evaluation in terms
of cell expectations. We illustrate the approach in a simulation
study and in a comparative evaluation of operational earthquake
forecasts for Italy.

\end{abstract}


\section{Introduction}

In many situations, scientific forecasts of uncertain future
quantities provide critical input to societally relevant decision
making.  For example, criminologists develop methods for forecasts of
criminal offences \citep{Mohleretal2011, Flaxmanetal2019,ZhuaMat2019},
epidemiologists assess when and where people catch diseases
\citep{MeyerHeld2014, Schoenbetal2019}, and seismologists use
statistical models to study and forecast earthquake behaviour
\citep{Ogata1988, Ogata1998, Zhuangetal2002, BraySchoen2013}.  The
relevant events in these examples --- criminal offences, infections,
and seismic events --- occur as random point patterns in space and
time.  In probabilistic terms they are modelled as realizations of
point processes \citep{DVJvol1}.  Beyond the development of new point
process models for these phenomena, there is a growing demand for
theoretically principled evaluation methods. 

Model evaluation and forecast assessment are subjects of a vast body
of scientific literature.  Among a plethora of approaches, a simple
distinction can be made between the assessment of absolute and
relative performance.  Evaluating absolute performance, or assessing
goodness-of-fit, means checking whether the assumed model is
consistent with the data and rejecting it if this is not the case.

If
two or more models are available, it is desirable to assess their 
relative performance and check whether a model outperforms its
competitors.  Consistent scoring functions and proper scoring rules
are widely used and well-studied tools that serve this purpose, see
e.g.\ \citet{GneitRaft2007} and \citet{Gneit2011}.  The central
objective of our paper is to demonstrate that this idea and associated
statistical methods transfer to point process forecasts and,
consequently, provide practical, yet theoretically principled tools
for comparative forecast evaluation in this setting.

A scoring function or scoring rule assigns a real number to each pair
of a forecast and the respective realized observation of a random
variable $Y$.  If the forecast is expressed as a statistical property,
such as the mean or a quantile of the (possibly, implicit) predictive
distribution, this mapping is called \textit{scoring function},
whereas the term \textit{scoring rule} is used when an entire
predictive distribution is reported.  In either case, the key
requirement to be satisfied is that forecasting the truth yields the
best score in expectation: A scoring function is \textit{consistent}
for a statistical property if the value of this property for a
distribution $F$ is a minimizer of the expected score with respect to
$F$.  Likewise, a scoring rule is \textit{proper} if the expected
score with respect to $F$ is minimized by forecasting $F$.  In
addition to forecast comparison, propriety and consistency allow for
regression and $M$-estimation \citep{GneitRaft2007}.

Thus far, statistical seismology has been a driving force in the
development of methods to evaluate point process models, see e.g.\
\citet{BraySchoen2013} for a review.  In particular, the regional
earthquake likelihood models (RELM) initiative \citep{Field2007} and
its successor, the Collaboratory for the Study of Earthquake
Predictability (CSEP) \citep{Zecharetal2010b, Schoretal2018}, have set
up forecast experiments for the prospective evaluation of models based
on a number of statistical tests.  \citet[p.~518]{BraySchoen2013}
point out the connection between some of these tests and the scoring
literature by stating that ``numerical tests such as the L-test, can
be viewed as examples of scoring rules [\ldots]".  The paper by 
\citet{Heinetal2019} makes this connection explicit and derives
consistent scoring functions to compare forecasts in the point process
setting.  We complement their simulation-based approach and develop an
alternative, computationally much less intense framework, in which we
work with distributional properties for which closed form expressions
under the posited point process model are available.  This yields a
flexible approach to forecast comparison, which incorporates existing
methods, and admits new perspectives on the strengths and weaknesses
of the CSEP methodology for earthquake forecast evaluation.

The remainder of the paper is structured as follows.
Section~\ref{sec:scoring} recalls fundamentals on scoring functions
and their role in forecast evaluation and model selection.
Section~\ref{sec:scoring4PP} rigorously introduces scoring functions
for point patterns and compares to the approach of
\citet{Heinetal2019}.  The use of consistent scoring functions for the
intensity is illustrated in finite sample simulation experiments in
Section~\ref{sec:simulations}.  In Section~\ref{sec:case} we evaluate
operational earthquake forecasts for Italy and discuss how scoring
functions relate to extant methods in seismology.  The paper closes
with a discussion in Section~\ref{sec:discussion}.

The main article concentrates on scoring functions for the intensity
-- the most fundamental first order property of a point process.
Scoring functions and simulation experiments for further standard
properties such as moment measures are addressed in the Supplementary
Material.

\section{Scoring functions and forecast evaluation}  \label{sec:scoring}

The following overview of consistent scoring functions and their
role in comparative forecast evaluation is primarily based on
\citet{Gneit2011}.

Let $\sO$ and $\sA$ be subsets of a real vector space, and let $\cF$
be a collection of probability distributions on the
Borel-$\sigma$-algebra of $\sO$.  We interpret $x \in \sA$ as a
forecast in terms of a single-valued \emph{functional}
$T: \cF \rightarrow \sA$ that is to be compared to an outcome in
$\sO$.  A function $S: \sA \times \sO \rightarrow \RR$ is called
\textit{scoring function} if for all $x \in \sA$ the mapping
$S(x,\cdot)$ is $F$-integrable for all $F \in \cF$.  The literature
usually distinguishes point forecasts ($A \subseteq \RR^n$) and
probabilistic forecasts ($A = \cF$ and $T$ is the identity) and uses
the term \textit{scoring rule} in the latter setting.  We do not make
this distinction and exclusively use the term scoring function.

The key concept which motivates the use of scoring functions is
consistency, meaning that a perfect forecast should achieve the lowest
score in expectation.  Specifically, a scoring function $S$ is
\textit{consistent} for a functional $T: \cF \rightarrow \sA $ if for
all $x \in \sA$ and $F \in \cF$ we have
\begin{align} \label{eq:consistency}
\E_F S(x, Y)  \geq  \E_F S(T(F), Y),
\end{align}
where the expectation $\E_F$ refers to the random variable $Y$
following the distribution~$F$.  It is \textit{strictly consistent}
for $T$ if in addition equality in~\eqref{eq:consistency} implies
$x = T(F)$.  A central question is which functionals $T$ are
\textit{elicitable}, i.e.\ possess a strictly consistent scoring
function.  Many elicitable functionals and corresponding classes of
strictly consistent scoring functions are known, e.g.\ expectations,
quantiles, and expectiles \citep{Gneit2011, DawidMusio2014,
  FronKash2015b, FronKash2021}.  For $\sA = \cF$ the most relevant
functionals are the identity and restrictions to the tails
\citep{GneitRaft2007, GneitRanjan2011, Lerchetal2017, HolzKlar2017}.

A fundamental result is that expectations of integrable functions are
elicitable.  For instance, $\E_F (x - Y)^2$ is uniquely minimized by
$x = \E_F Y$, thus the \textit{quadratic} score $S(x,y) = (x-y)^2$ is
a strictly consistent scoring function for the expectation functional.
To state a general theorem on the elicitability of expectations
\citep{Savage1971, Gneit2011, FronKash2015b}, let
$\sA, \sO \subseteq \real^k$ and let $\nabla f(x)$ denote the
subderivative of a convex function $f: \sA \rightarrow \real^k$ at
$x \in \real^k$.  The subderivative or subgradient is a generalization
of the derivative that applies to any convex function, and the two
concepts coincide if the derivative exists \citep{Rockafellar1970}.
The function $b : \sA \times \sO \rightarrow \real$ defined by
\begin{align}  \label{eq:Bregman}
b(x,y) = - f(x) - \nabla f(x)^\top ( y -  x)
\end{align}
is called a \textit{Bregman function} for $f$.  If $f$ is strictly
convex, we call $b$ \textit{strict}.

\begin{theorem}[elicitability of expectations]
\label{thm:ElicitExpect}
Let $h : \sO \to \real^k$ be $F$-integrable for all $F \in \cF$.  Then
the functional $T : \cF \rightarrow \sA \subseteq \real^k$ defined via
\begin{align*}
T(F)  = \int h(y) \dd F(y) = \E_F \, h(Y)  
= \left( \E_F  h_1 (Y), \ldots, \E_F  h_k (Y) \right)^\top
\end{align*}
is elicitable, and consistent scoring functions
$S : \sA \times \sO \rightarrow \real$ are given by
$S(x,y) = b(x, h(y))$, where $b$ is a Bregman function.  If $b$ is
strict, then $S$ is strictly consistent for~$T$.
\end{theorem}

In general, bijective transformations of the domain $\sA$ preserve the
elicitability of a functional, a fact which is usually called
\textit{revelation principle} \citep[Theorem~4]{Gneit2011}.  Likewise,
if we consider transformations of the observation domain $\sO$, we can
state the following simple result, which resembles, but differs from,
findings on weighted functionals as discussed in
\citet{GneitRanjan2011} and \citet[Theorem~5]{Gneit2011}.  The proof
is a straightforward consequence of integration with respect to the
pushforward measure and thus omitted.

\begin{prop}[transformation principle] \label{prop:TrafoPrinciple} Let
  $T: \cF \to \sA$ be an elicitable functional and
  $S: \sA \times \sO \to \real$ a (strictly) consistent scoring
  function for $T$.  Let $g: \sO' \rightarrow \sO$ be measurable, and
  let $\cF'$ be a set of distributions on $\sO'$, such that
  $\{ F' \circ g^{-1} \mid F' \in \cF' \} \subseteq \cF$.  Then the
  functional $T': \cF' \to \sA$ defined via
  $T'(F') := T( F' \circ g^{-1} )$ is elicitable with (strictly)
  consistent scoring function $S'(x,y) = S(x, g(y))$.
\end{prop}

In case multiple forecasts in terms of an elicitable functional $T$
are available, their predictive performance can be assessed in a
natural way:  If $S$ is a strictly consistent scoring function for
$T$, then a forecast is considered superior to its competitor if it
achieves a lower expected score with respect to $S$.  This allows for
a choice between two forecasts based on their difference in expected
scores, without further assumptions on the data-generating process.

To illustrate the idea, we introduce a simple point process scenario,
which is motivated by our earthquake forecasting case study
(Section~\ref{sec:case}).  Let $\Phi$ be a spatial point process which
models the locations of earthquake epicentres in a specified region
during a period of seven days.  Let $S$ be a scoring function such
that $S(r, \Phi)$ is the score of the forecast report $r \in \sA$, and
assume that $S$ is strictly consistent for a statistical property of
point processes, e.g.\ the intensity measure (see
Section~\ref{subsec:Intensity}).  In this situation, two intensity
forecasts $r$ and $r^*$ can be compared based on
$\E \left( S(r, \Phi) - S(r^*, \Phi) \right)$, where, due to the
consistency of $S$, negative values support $r$, while positive values
support $r^*$.

In typical applications we face forecasts $r_t$, $r_t^*$ and
corresponding realizations $\Phi_t$ of the point process for time
points $t = 1, \ldots, N$.  With these values, the expected score
difference can be estimated via the realized average score difference.
Substantial deviations from zero then indicate differences in the
predictive performance of the forecast sequences $(r_t)$ and
$(r^*_t)$.  To estimate the uncertainty inherent in the score
differences it is common to use the Diebold--Mariano test
\citep{DiebMari1995} or extensions of this testing framework, see
e.g.\ \citet{NoldeZieg2017} and \citet{HerGent2011}.

Although we here focus on the specific scenario of a discretely
observed spatial point process, strictly consistent scoring functions
can be used in many other point process settings, as discussed in
Section~S1 of the Supplementary Material.

\section{Consistent scoring functions for point patterns}  \label{sec:scoring4PP}

We now turn our attention to the situation, where each observation is
a finite point pattern.  We first connect to existing theory
(Section~\ref{sec:scoring}) and then derive scoring functions for the
distribution and the intensity measure. Scoring functions for further
point process characteristics are discussed in Section~S2 of the
Supplementary Material.

\subsection{Technical context}

We follow the common convention that a finite \textit{point process}
$\Phi$ is a random element in the space $\Mz = \Mz(\cX)$ of finite
counting measures on the Borel set $\cX \subseteq \RR^d$ and refer to
\citet{DVJvol1} for details.  We denote a set of probability measures
on $\Mz$ by $\cP$ and the distribution of $\Phi$ by $P_\Phi$.  Any
forecast is issued for a \textit{functional}
$\Gamma : \cP \rightarrow \sA$ and is to be compared to an outcome in
$\Mz$.  We call a mapping $S : \sA \times \Mz \to \real$ a
\textit{scoring function} if
$\E_P S(a, \Phi) = \int S(a,\varphi) \dd P(\varphi)$ exists for all
$a \in \sA$ and $P \in \cP$.  \textit{Elicitability} of $\Gamma$ as
well as \textit{(strict) consistency} of $S$ is then defined as above
via inequality~\eqref{eq:consistency}, i.e.\ $S$ is strictly
consistent for $\Gamma$ if
$\E_P S(a, \Phi) \geq \E_P S(\Gamma(P), \Phi)$ for all $a \in A$ and
$P \in \cP$ and equality implies $a = \Gamma(P)$.  For ease of
presentation and practical implementation, we will usually state how
the score of a realization
$\varphi = \sum_{i = 1, \dots, n} \delta_{y_i} \in \Mz$ is computed
from an enumeration of its points, i.e.\ from the set
$\{ y_1, \ldots, y_n \}$, where $n = \vert \varphi \vert$ is the total
mass of the counting measure $\varphi \in \Mz$.   To make this
meaningful, we will ensure that for spatial processes all scoring
functions are independent of the enumeration of points
\citep[Chapter~5]{DVJvol1}.

In light of Theorem~\ref{thm:ElicitExpect}, constructing simple
examples for elicitable functionals of point processes is
straightforward: Point processes induce real-valued random variables
in many ways and the expectations of these random variables (provided
they are finite) will be elicitable functionals.

\begin{example}[expected number of points]
\label{ex:SimpleProperty}
Given a set $B \in \mathcal{B}(\cX)$, the ($\natural_0$-val\-ued) random
variable $\Phi(B)$ denotes the number of points of $\Phi$ in $B$.
According to Theorem~\ref{thm:ElicitExpect} the functional
$\Gamma_B: \cP \to \real$ given by $\Gamma_B (P)= \E_P \Phi (B)$ is
elicitable with Bregman scoring function
\begin{align*}
S_B(x, \varphi) = b(x, \varphi(B)) = - f(x) - \nabla f(x)^\top 
(\varphi(B) -  x),
\end{align*}
where $f : [0, \infty) \to \real$ is a strictly convex function.
\end{example}

This construction is not limited to the expected number of points in a
set, but works for any combination of elicitable functional (e.g.\
expectation) and point process feature (e.g.\ number of points): Let
$\sO$ be an observation domain and $g : \Mz \to \sO$ a measurable
mapping.  The transformation principle
(Proposition~\ref{prop:TrafoPrinciple}) then implies that the
functional $\Gamma (P) := T (P \circ g^{-1})$ is elicitable whenever
$T : \lbrace P \circ g^{-1} \mid P \in \cP \rbrace \to \sA$ is
elicitable.  We recover Example~\ref{ex:SimpleProperty} by choosing
$T(F) = \E_F Y$ and $g(\varphi) = \varphi (B)$.  The elicitability of
other ``simple" properties such as finite-dimensional distributions
and void probabilities is a straightforward consequence of
Proposition~\ref{prop:TrafoPrinciple} and deferred to the
Supplementary Material.

Different choices for $T$ and $g$ in
Proposition~\ref{prop:TrafoPrinciple} lead to a wide variety of
different functionals and consistent scoring functions.  The core idea
in \citet{Heinetal2019} is to choose $T$ as the identity on
$\lbrace P \circ g^{-1} \mid P \in \cP \rbrace$.  Two distributional
models $P, Q \in \cP$ of the process $\Phi$ can then be compared based
on realizations by comparing $P \circ g^{-1}$ and $Q \circ g^{-1}$ via
a consistent scoring function for distributions.  The mapping
$g: \Mz \to \sO$ is selected to be an estimator of some quantity of
interest, e.g.\ a kernel-based intensity estimator.  Since the
distributions of such estimators will usually not be explicitly
available, approximating the scoring functions via simulations becomes
necessary.  Moreover, as different $P \in \cP$ may lead to the same
law $P \circ g^{-1}$, this approach hinges on the ability of $g$ to
discriminate between two distributions $P$ and~$Q$.

Instead of following this approach, we focus on common point process
characteristics $\Gamma : \cP\to \sA$ and develop strictly consistent
scoring functions for them.  This allows for a direct comparison of
the characteristic $\Gamma$, which includes distributional models
$P \in \cP$ as a special case. In contrast, comparison in
\citet{Heinetal2019} always depends on specific aspects of the
distributions in $\cP$ which are determined via the estimator choice
$g$.  This arguably leads to a good discrimination ability, as the
whole point process distribution is taken into account, whereas
comparison in our approach focuses on how similar the property values
$\Gamma (P)$ and $\Gamma (Q)$ (e.g.\ the intensity measures) are.
However, this also means that knowledge of the distribution $P$ is not
needed in our setting, as long as $\Gamma(P)$ is available.  In cases
where $\Gamma$ can be computed explicitly for models in $\cP$, this
avoids point process simulations, which might be prohibitive in
routine applications.  Furthermore, this simplifies reporting, since
forecasters do not need to come up with a fully specified point
process distribution.  For these reasons the methodology proposed here
complements the approach developed by \citet{Heinetal2019}, and which
is more suitable depends on the setting at hand.

\subsection{Distribution and density}  \label{subsec:Density_general}

In this subsection we construct consistent scoring functions for the
identity functional $\Gamma = \mathrm{id}_\cP$, i.e.\ for the entire
point process distribution.  To this end we need to specify how we
represent the law $P_\Phi$ of the finite point process $\Phi$ on
$\cX$.  One way to do so is via sequences $(p_k)_{k \in \natural_0}$
and $(\Pi_k)_{k \in \natural}$.  Each $p_k$ specifies the probability
of finding $k$ points in a realization, and $\Pi_k$ are symmetric
probability measures on $\cX^k$ which describe the distribution of any
ordering of points, given $k$ points are realized
\citep[Chapter~5.3]{DVJvol1}.  Although this representation already
allows for the construction of consistent scoring functions for
$P_\Phi$, we focus on the case where densities are available, since
these are often more convenient to deal with, especially when
multivariate distributions are of interest.

\citet{GneitRaft2007} formalize density forecasting as follows:  Let
$(\Omega , \cA,$ $\mu)$ be a $\sigma$-finite measure space and for
$\alpha > 1$ let $\cL_\alpha$ consist of all (equivalence classes of)
densities $p$ of probability measures $P$ that are absolutely
continuous with respect to $\mu$ and such that $ \Vert p \Vert_\alpha
:= \left( \int_{\Omega} p(\omega)^\alpha  \dd \mu (\omega )
\right)^{1/ \alpha}$ is finite.  In this setting, important examples
of strictly consistent scoring functions $S: \cL_\alpha \times \Omega
\to\real$ are the \textit{pseudospherical} and the
\textit{logarithmic} score, defined via
\begin{align}   \label{eq:Pseudo_and_Log}
\mathrm{PseudoS} (p, \omega) = - p(\omega)^{\alpha -1} / 
\Vert p \Vert_\alpha^{\alpha -1} 
\quad \text{ and } \quad
\mathrm{LogS}(p, \omega) = - \log p(\omega) ,
\end{align}
respectively. The logarithmic score is the (appropriately scaled)
limiting case of the pseudospherical score as $\alpha \to 1$.

Returning to point processes we follow \citet[Chapters~5.3 
and~7.1]{DVJvol1} and let $P_0$ denote the distribution of the Poisson
point process with unit rate on some bounded domain
$\cX \subset \real^d$.  If $P \in \cP$ is absolutely continuous with
respect to $P_0$, then the Radon-Nikod\'ym density
$\mathrm{d}P / \mathrm{d}P_0$ exists and can be regarded as the
density of $P$.  It can be computed via the identity
\begin{align*}
\frac{\mathrm{d}P}{\mathrm{d}P_0} (\varphi) = \exp( \vert \cX \vert ) \, \frac{j_k (y_1, \ldots, y_k)}{k!},
\end{align*}
where $\vert \cX \vert$ denotes the Lebesgue
measure of $\cX$, $y_1, \ldots, y_k$ are the points of
$\varphi \in \Mz$, and the (symmetric) function $j_k$ given by
\begin{align}  \label{eq:Janossy}
j_k (x_1, \ldots, x_k) \dd x_1 \cdots \dd x_k = k! p_k \dd 
\Pi_k (x_1, \ldots, x_k)
\end{align}
is the $k$-th \textit{Janossy density} of $\Phi$. For $k=0$
this is interpreted as $j_0 = p_0$. The value $j_k (x_1, \ldots, x_k)$
can be understood as the \textit{likelihood} of $k$ points
materializing, one of them in each of the distinct locations $x_1,
\ldots, x_k \in \cX$.

In principle, plugging the Janossy densities
into~\eqref{eq:Pseudo_and_Log} allows us to obtain scoring functions
for the point process distribution $P$.
However, two important difficulties need to be addressed in the point
process setting.  First, explicit expressions for
$(j_k)_{k \in \natural_0}$ are usually hard to determine and known
only for some models, see \citet[Chapter~7.1]{DVJvol1} and
Example~\ref{ex:PoissonLikelihood} below.  Second, even if explicit
expressions are available, calculating the norm
$\Vert \mathrm{d} P / \mathrm{d} P_0 \Vert_\alpha$ amounts to
computing
$ (k!)^{-1} \int j_k(x_1, \ldots, x_k)^\alpha \dd x_1 \ldots \dd x_k$
for all $k \in \natural$, which may be prohibitive.  This complicates
the use of scoring functions relying on $\Vert \cdot \Vert_\alpha$,
such as the pseudospherical score~\eqref{eq:Pseudo_and_Log}.  We will
thus only consider the logarithmic score here, and discuss a further
choice in the Supplementary Material.

Assume that for all distributions $Q \in \cP$ the corresponding
Janossy densities 
$(j_k^Q)_{k \in \natural_0}$ are well-defined.  Due to the strict
consistency of the logarithmic score, the
function $S: \cP \times \Mz \to \real$ defined via
\begin{align}   \label{eq:PPlogScore}
S( (j_k^Q)_{k \in \natural_0}, \{y_1, \ldots, y_n\} ) 
= - \log (j_n^Q (y_1, \ldots , y_n) )
\end{align}
for $n \in \natural$ and
$S( (j_k^Q)_{k \in \natural_0}, \emptyset ) := - \log (j_0^Q)$ is a
strictly consistent scoring function for the distribution of the point
process $\Phi$.  The term $- \vert \cX \vert + \log (n!)$ can be
omitted, since it is independent of the forecast report
$(j_k^Q)_{k \in \natural_0}$.  This choice recovers the log-likelihood
of the point process distribution $Q$ from the perspective of
consistent scoring functions.

\begin{example}[Poisson point process]  \label{ex:PoissonLikelihood}
Let $\Phi$ be an inhomogeneous Poisson point process with intensity
$\lambda : \cX \to [0, \infty)$.  It is well-known that $\Phi$ admits
the densities
\begin{align*}
j_{n, \lambda} (y_1, \ldots, y_n) = \Big( \prod_{i=1}^{n} 
\lambda (y_i) \Big)  \exp \left( - \int_\cX \lambda (y) \dd y \right)
\end{align*}
for $n \in \nat$.  In case $n = 0$ the product is interpreted as one.
When reporting the Poisson point process distribution $P_\Phi$ via its
Janossy densities, \eqref{eq:PPlogScore} gives the score
\begin{align}  \label{eq:Poisson_logScore}
S( P_\Phi, \{y_1, \ldots, y_n\} ) =  - \sum_{i=1}^{n} \log 
\lambda (y_i) + \int_\cX \lambda (y)  \dd y 
\end{align}
for $n \in \natural$ and $S(P_\Phi, \emptyset) = \int_\cX \lambda (y) \dd y$.
\end{example}

Before turning to the intensity measure, we briefly discuss temporal
point processes, which demand a special treatment since the dimension
``time" possesses a natural ordering.  The instantaneous rate of
points occurring in the point process $\Phi$ is usually described via
the \textit{conditional intensity}
\begin{align}  \label{eq:condintensI}
\lambda^* ( t ) = \lim_{\Delta t \to 0} \frac{ \E \left( \Phi ( 
( t, t + \Delta t) ) \mid \cH_t \right)}{\Delta t} ,
\end{align}
where $(\cH_t)_{t \in \real}$ is the filtration generated by the
history of $\Phi$ \citep[Chapter~7]{Rein2018,DVJvol1}. Although 
$\lambda^*(t)$ is random, it is deterministic conditional on $\Phi$,
thus a measurable mapping linking it to $\Phi$ allows for modelling as
well as evaluation via consistent scoring functions.

Specifically, let $\Phi$ be a point process on $\real$ and consider an
observation window $\cX := [0,T]$ for some $T >0$.  Given a
realization $0 < t_1 < \ldots < t_n$ of $\Phi$ the realized values of
the conditional intensity can be computed for all $t \in \cX$.  More
precisely, for a $t\in \cX$ with $t_1 < \ldots < t_i \le t < t_{i+1}$
we denote the realized value of $\lambda^*$ at $t$ via
$\lambda^*(t \mid t_1, \ldots, t_i )$.  Since the collection of all
mappings $t \mapsto \lambda^*(t \mid t_1, \ldots, t_i)$ for all
$i=0, \ldots, n$ and all possible realizations $t_1, \ldots, t_n$
uniquely determines the distribution of $\Phi$ \citep{DVJvol1},
comparing forecasts for the conditional intensity is equivalent to a
comparison of forecasts for the distribution.  This connection is made
explicit by the representation of the likelihood of
$t_1 , \ldots, t_n$ occurring in $[0,T]$ via
\begin{align}    \label{eq:temporal_likelihood}
j_n (t_1, \ldots, t_n ) = \Big( \prod_{i=1}^{n} 
\lambda^* (t_i) \Big) \exp \left( - \int_{0}^{T} \lambda^* (u) 
\dd u \right) ,
\end{align}
where the product is interpreted as one if no points occur.
Consequently, (strictly) consistent scoring functions for the
conditional intensity can be obtained by arguments similar to above.

\begin{example}[Recovery of log-likelihood of a temporal point process] \label{ex:condIntensity_log}
  Plugging~\eqref{eq:temporal_likelihood} into the logarithmic
  score~\eqref{eq:PPlogScore} we see that the scoring function
\begin{align*}
S(\lambda^*, \{t_1, \ldots, t_n \} ) = - \sum_{i=1}^{n}  \log 
\left( \lambda^* (t_i) \right) + \int_{0}^{T} \lambda^* (u ) \dd u ,
\end{align*}
is strictly consistent for the conditional intensity. This recovers
the log-likelihood of a temporal point process
\citep{DVJvol1,Rein2018}. If $\Phi$ is a Poisson point process on
$\real$, its conditional intensity $\lambda^*$ agrees with its
intensity $\lambda$, and $S$ coincides
with~\eqref{eq:Poisson_logScore}.
\end{example}

\subsection{Intensity measure}  \label{subsec:Intensity}

One of the key characteristics of a point process $\Phi$ is its
intensity measure $ \Lambda: B \mapsto \E \Phi (B)$ that quantifies
the expected number of points in any set $B \in \cB(\cX)$
\citep{DVJvol1,Chiuetal2013}.  Analogous to the first moment of a
univariate random variable, it describes the average behaviour of the
point process $\Phi$.  For a fixed Borel set $B$, we have already
identified the expected number of points $\Lambda(B)=\E \Phi (B)$ as
an elicitable functional (Example~\ref{ex:SimpleProperty}).  Here we
focus on constructing scoring functions for the full measure $\Lambda$
as a functional on $\cP$ with values in a set of finite measures $\Mf$
on $\cX$. To this end, we call
$\Lambda^* := \Lambda / \vert \Lambda \vert$, where
$\vert \Lambda \vert := \Lambda(\cX)$ is the total mass of $\Lambda$,
the \textit{normalized measure} of a finite measure $\Lambda \in \Mf$.

\begin{prop}  \label{prop:NormalizedIntens}
Set $\cF := \{ \Lambda^* \mid \Lambda \in \Mf \}$ and let $S' : \cF
\times \cX \rightarrow \real$ be a (strictly) consistent scoring
function for $\mathrm{id}_\cF$. Let $b : [0, \infty) \times 
[0, \infty) \rightarrow \real$ be a (strict) Bregman
function, as in~\eqref{eq:Bregman}. The scoring function
$S : \Mf \times \Mz \to \real$ defined via
\begin{equation*}
S(\Lambda , \lbrace y_1, \ldots, y_n \rbrace) :=  \sum_{i=1}^{n} 
S'(\Lambda^*, y_i)  + c  b \big( \vert \Lambda \vert, n \big) 
\end{equation*}
for $n \in \natural$ and $S(\Lambda, \emptyset) = c b(\vert \Lambda
\vert , 0)$ for $c > 0$, is consistent for the intensity measure. It is
strictly consistent if $S'$ is strictly consistent and $b$ is strict.
\end{prop}

\begin{proof}
Let $W \in \Mf$ and $\Phi$ be a point process with intensity measure
$\Lambda \in \Mf$ and distribution $P \in \cP$. The difference in
expected scores is 
\begin{align*}
\E_P \left[ S(W, \Phi) - S(\Lambda, \Phi) \right]
&= \int \sum_{x \in \varphi} S'(W^*, x) - S'(\Lambda^*, x) 
\dd P (\varphi)  \\
&\phantom{=}+ c \E_P \left( b( \vert W \vert, \vert \Phi \vert) - b(\vert \Lambda \vert, \vert \Phi \vert ) \right) 
\end{align*}
and the last term is nonnegative since $b$ is a Bregman
function. Using Campbell's theorem, the second expression equals
\begin{align*}
\int_\cX S'(W^*, x) -S'(\Lambda^*, x) \, \mathrm{d} \Lambda (x) 
= \vert \Lambda \vert \int_\cX S'(W^*, x) -S'(\Lambda^*, x) \dd 
\Lambda^* (x) ,
\end{align*}
and is also nonnegative, due to the consistency of $S'$. If the score
difference is zero, $b$ is strict, and $S'$ is strictly consistent,
this gives $W^* = \Lambda^*$ and
$\vert W \vert = \vert \Lambda \vert$, showing that $S$ is strictly
consistent for the intensity measure. 
\end{proof}

In principle, it is possible to define scoring functions which only
depend on normalized measures, by using arguments in
\citet{HendBueh1971} who discuss a connection to homogeneous functions
on the cone induced by a set of probability measures.  As we are
interested in the full intensity measure, we combine the total mass
$\vert \Lambda \vert = \E \Phi (\cX)$, which is an elicitable property
of $\Phi$ (Example~\ref{ex:SimpleProperty}), with $\Lambda^*$ to
obtain a consistent scoring function.

\begin{example} \label{ex:intens_log_and_quad} As an important special
  case, assume that each $\Lambda \in \Mf$ admits a density $\lambda$
  with respect to Lebesgue measure.  Using the common quadratic score
  for $b$ and the logarithmic score~\eqref{eq:Pseudo_and_Log} for
  $S'$, the strictly consistent scoring function of
  Proposition~\ref{prop:NormalizedIntens} becomes
\begin{align*}
S(\Lambda, \{ y_1, \ldots, y_n\} ) = - \sum_{i=1}^{n} 
\log (\lambda (y_i) ) + n \log \vert \Lambda \vert + c \, 
( \vert \Lambda \vert - n)^2
\end{align*}
for some $c > 0$. Simulation experiments in
Section~\ref{sec:simulations} illustrate how $S$ can be used to
compare intensity forecasts.
\end{example}

The choice of the constant $c > 0$ in
Proposition~\ref{prop:NormalizedIntens} is irrelevant for (strict)
consistency of the scoring function $S$.  However, since $S$ evaluates
both the shape and the total mass of the intensity, judicious choices
of $c$ serve to balance the scoring components.

\section{Simulation study} \label{sec:simulations}

In this section we investigate finite sample properties of scoring
function-based model evaluation via mean score differences, with focus
on intensity forecasting for spatial point processes.  All
calculations are performed with \textsf{R} \citep{RCoreTeam},
including point process simulations with the \texttt{spatstat} package
\citep{BaddTurner2005, Baddeleyetal2015}.

We compare different intensity reports for a point process $\Phi$ on
the window $[0,1]^2$ based on $N \in \nat$ realizations, where $N$
could reflect a number of different time windows, e.g.\ $N = 52$ for
one year of weekly data.  We draw $N = \Nsim$ i.i.d.\ samples
$\varphi_i$ from $\Phi$ and use the mean score
\[
\bar s_j := \frac 1N \sum_{i=1}^{N} S(f_j, \varphi_i)
\]
as an estimator of the expected score $\E S(f_j, \Phi)$ of a given
forecast intensity $f_j$ in the population.  We use the scoring
function $S$ from Example~\ref{ex:intens_log_and_quad} with scaling
factor $c = 1 / 10$ such that the logarithmic and squared terms vary
at the same order of magnitude.  The simulations are repeated $M =
\Msim$ times to assess the variation in mean scores.

We consider four different data-generating processes for $\Phi$, all
of which have (approximate) intensity
$f_0(x,y) = 6 \sqrt{x^2 + y^2}$, which leads to four different
simulation experiments.  In the first experiment $\Phi$ is an
inhomogeneous Poisson point process.  In the second $\Phi$ is a
determinantal point process (DPP) with Gaussian covariance such that
its points exhibit moderate inhibition.  In the remaining two
simulation experiments $\Phi$ inclines to clustering.  For the third
one, we choose a log-Gaussian Cox process (LGCP) with exponential
covariance and log-expectation $\mu$ such that its intensity equals
$f_0$.  In the last experiment $\Phi$ is an inhomogeneous Thomas
process, i.e.\ a cluster process which arises from an inhomogeneous
Poisson process as parent and a random number of cluster points which
are drawn from a normal distribution centered at its parent point.
Due to this clustering, the intensity of the Thomas process is only
approximately equal to $f_0$.  For details on the processes see
\citet{Lavetal2015}, \citet[Chapter~6]{Illianetal2008} and Section~S3
of the Supplementary Material.

The study compares six different intensity forecasts, namely, $f_0$ and
\begin{align*}
f_1 (x,y) &= 7.8 \sqrt{(x-0.2)^2 + (y-0.1)^2}, \\
f_2 (x,y) &= 2.3 (x + 3y), \\
f_3 (x,y) &= 10  \sqrt{(x-0.2)^2 + (y-0.1)^2}, \\
f_4 (x,y) &=  7.5 \exp \left[ - 3  \left\{ \left(x- 0.6 \right)^2 + \left(y- 0.6 \right)^2 \right\} \right], \\
f_5 (x,y) &= 2 \left\{ \frac{1}{\sqrt{1.2 - x}} + 2(1-y)  \right\}.
\end{align*}
These choices are motivated as follows.  Intensity $f_1$ has the
correct shape, up to a small shift, and $f_3$ is a version of $f_1$
with too high total mass.  Intensity $f_2$ is similar to $f_0$ but
linear, while $f_4$ and $f_5$ have completely different shape, as
illustrated by Figure~\suppIntensities\ in the Supplementary Material.
Except for $f_3$, all intensities put roughly identical mass on
$[0,1]^2$.  This allows for an assessment of how the scoring function
reacts to misspecifications in shape instead of total mass.

Figure~\ref{fig:simALL4} shows the mean score differences between
the five different forecasts $f_1, \ldots, f_5$ and the optimal
forecast $f_0$ for all experiments.  The four experiment show a
similar pattern, namely $f_1$ is close to the optimal forecast, $f_2$
and $f_3$ less so, and the mean score differences of the misspecified
functions $f_4$ and $f_5$ are far from zero.  The fourth experiment
shows an increase in variance which likely stems from the strong
clustering tendency of the process.  Moderate clustering or
inhibition, as present in the third and second experiment, seem to
have almost no impact on the score differences.  Overall, varying the
intensity forecasts leads to pronounced differences in realized
average scores, highlighting differences in forecast performance.
Further experiments with different scoring functions as well as tests
for superior predictive ability are given in Section~S3 of the
Supplementary Material.

\begin{figure}[hbt]
\centering
\includegraphics[clip, width = \textwidth, trim = 0 10 0 0]{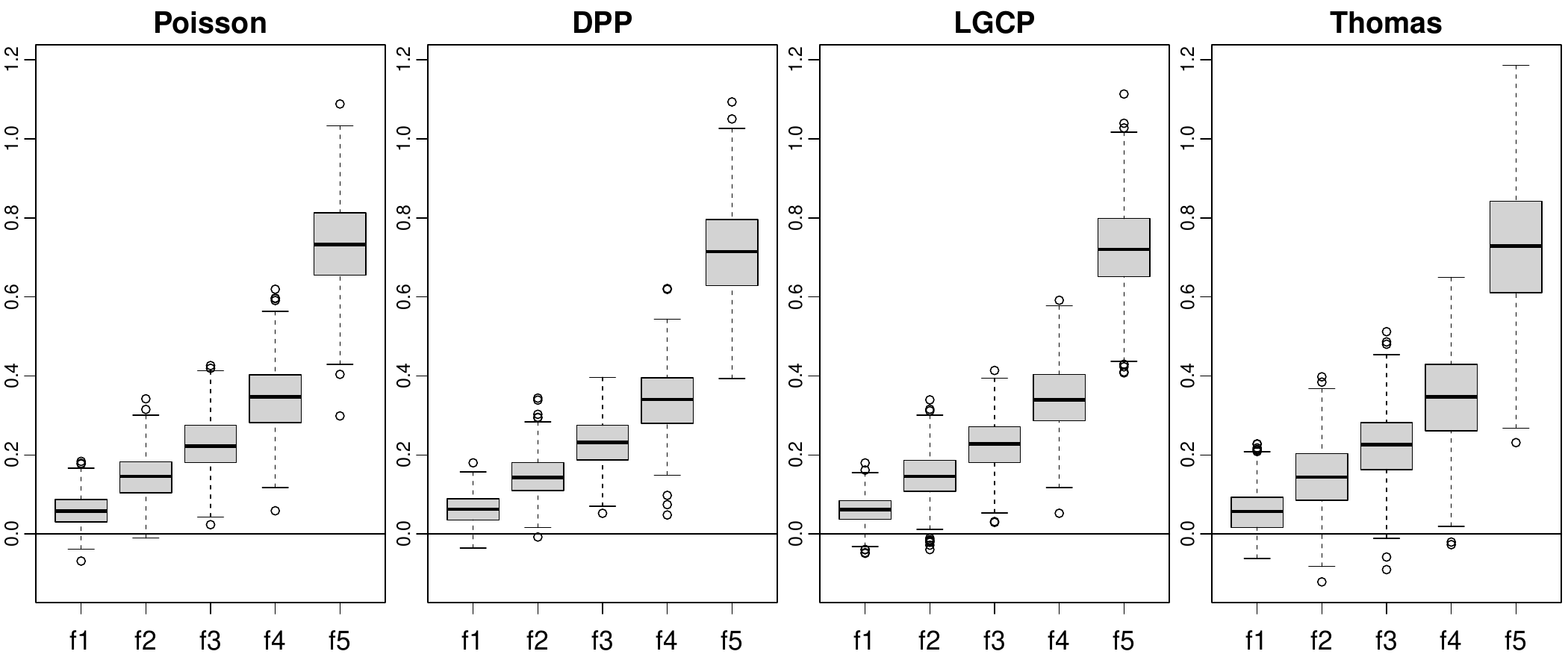}
\caption{Boxplots of the difference in mean scores $\bar s_j - \bar s_0$
for $j = 1, \ldots, 5$ and the scoring function $S$ from
Example~\ref{ex:intens_log_and_quad}.  From left to right, $\Phi$ is a
Poisson point process, a Gaussian determinantal point process, a
log-Gaussian Cox process, and an inhomogeneous Thomas process.  Means
are based on $N = \Nsim$ realizations, boxplots on $M = \Msim$
replicates.
\label{fig:simALL4}}
\end{figure}

\begin{figure}[htb]
\centering
\includegraphics[clip, width = 0.49\textwidth, trim = 0 0 0 20]{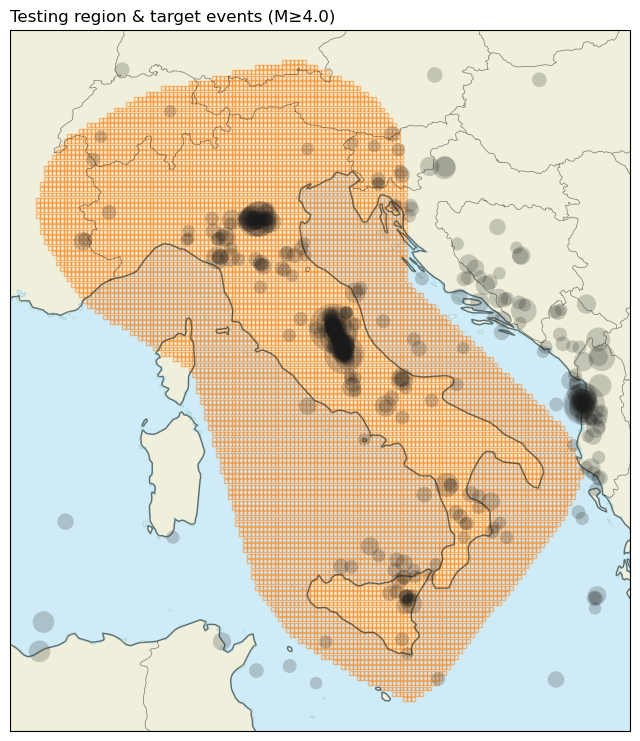}
\caption{Testing region of the Italian CSEP experiment. Gray circles represent locations of M4+ earthquakes. Figure reproduced from \citet{HerrMarz2023}.
\label{fig:testing_region}}
\end{figure}

\section{Case study: Earthquake forecasting}  \label{sec:case}

In this case study we illustrate how consistent scoring functions can
be used to compare earthquake forecasting models, and we shed new
light on extant evaluation methods in seismology.  All calculations
are performed with \textsf{R} \citep{RCoreTeam}.

\subsection{Earthquake forecasting experiments}

Over the past decades it has become consensus that earthquake
forecasts ought to be probabilistic, i.e.\ instead of specifying
whether or not an earthquake will occur, they provide a respective
predictive distribution or aspects thereof \citep{Jordanetal2011}.
Satistical models to issue such forecasts are based on spatio-temporal
point processes.  They are usually specified via a conditional
intensity (see~\eqref{eq:condintensI}) that exhibits self-exciting
behaviour, reflecting the conjecture that earthquakes trigger each
other and cluster in space and time. 
An important example is the
epidemic-type aftershock sequence (ETAS) model, see e.g.\
\citet{KagKnop1987} and \citet{Ogata1988,Ogata1998}.

The Collaboratory for the Study of Earthquake Predictability (CSEP, see Introduction)
evaluates earthquake forecasts prospectively in several regional
testing centers with standardized testing routines.  The prospective
approach uses only forecasts submitted in real time before the
respective outcomes are realized, which guarantees independence of the
forecasts from actual observations.  An important part of these
routines is the earthquake likelihood model testing approach of
\citet{KaganJack1995} and \citet{Schoretal2007}, which we discuss in
Section~\ref{subsec:RELM}.  Our case study relies on data from the
operational earthquake forecasting system in Italy (OEF-Italy,
\citet{Marzocchietal2014}), which is based on the three independent
short-term forecasting models that were tested prospectively in a
CSEP testing center for the Italian testing region
\citep{Taronietal2018}.  See Figure~\ref{fig:testing_region} for an
illustration.

The three independent models comprise LM \citep{LombMarz2010} and FMC
\citep{Falconeetal2010}, which are ETAS-based models with distinct
structure and calibration choices, and LG \citep{Woessneretal2010},
which is based on the short-term earthquake probability (STEP) model
of \citet{Gerstetal2005} and composed of sub-models.  We refer to the
original references for more details about the individual models.
OEF-Italy also includes an aggregated or ensemble forecast, namely, SMA, which
predicts a weighted average of the above three models using the
score model averaging (SMA) rule \citep{Marzetal2012}, with models being weighted
 inversely proportional to the log-likelihood of observed
data.  The SMA model is updated continuously based on new observations and
was successfully applied to track the evolution of
the recent earthquake sequence in central Italy in real time 
\citep{Marzocchietal2017}.

Our study considers earthquakes of magnitude greater or equal to four
(M4+) between April 2005 and May 2020 (5520 days) that fall into the
Italian CSEP testing region (Figure~\ref{fig:testing_region}).  The
testing region is divided into 8993 grid cells.  On each day, the four
models produce forecasts for the expected number of M4+ earthquakes in
the subsequent seven-day period for each grid cell.  The forecasts are
thus nonnegative values $x_{i,t}^{(j)}$ where $j$ denotes the model,
$i$ the cell, and $t$ the day.  They can then be compared to the
observed number of events in each cell for that upcoming week.  Since
the forecasts concern seven-day periods, this number is only known
seven days after a forecast was issued.  For the same reason, the
number of days available for evaluation reduces to 5514.

\begin{figure}[htb]
\centering
\includegraphics[width = .9\textwidth]{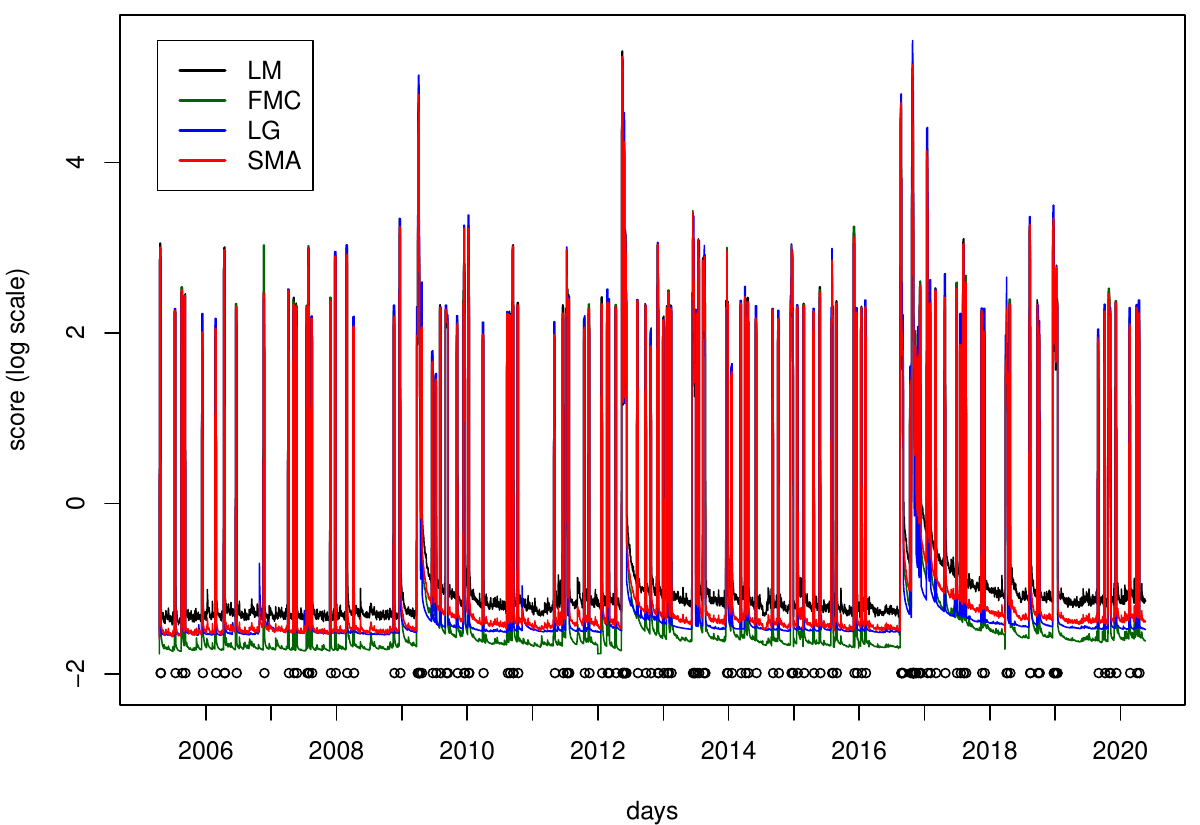}
\caption{Daily scores $s_{j,t}$ from \eqref{eq:realizedScores} based on $S_\mathrm{pois}$ for the four forecasting models from 2005 to 2020, logarithmic scale.  The circles indicate the days of M4+ earthquakes and the tickmarks on the horizontal axis mark the first day of each year.
\label{fig:plot_pois_time}}
\end{figure}

\subsection{Model comparison and results}  \label{subsec:model_comp}

Since the models we consider produce mean forecasts, we have to
employ (strictly) consistent scoring functions for the expectation
functional for a sound comparison, see also
Example~\ref{ex:SimpleProperty}.  Such functions are of the Bregman
form~\eqref{eq:Bregman} and a natural choice is the quadratic score
$S_\mathrm{quad} (x,y) = (x-y)^2$.  However, the quadratic score
focuses on no particular forecast cases in the sense of elementary
scores \citep{Ehmetal2016}.  As an alternative that puts more
emphasis on small forecast values and connects to the CSEP methods
(see Section~\ref{subsec:RELM}) we use the \textit{Poisson} scoring
function $S_\mathrm{pois} : (0, \infty) \times \nat_0 \to \real$
defined via
\begin{align}   \label{eq:PoisScore}
S_\mathrm{pois}(x, y) = -y \log (x) + x .
\end{align}
It is strictly consistent since it is a Bregman function corresponding
to the strictly convex function $f(x) = x (\log (x) - 1)$.  Note
that~\eqref{eq:PoisScore} can be interpreted as a discrete analogue to
the Dawid-Sebastiani-score \citep{DawidSebas1999}, but with the normal
distribution replaced by the Poisson distribution \citep{Brehmer2021}.
To obtain a daily score of the forecast models, the individual scores
for the 8993 grid cells are summed up.  The daily scores and the mean
score of model $j$ are thus given by
\begin{align}  \label{eq:realizedScores}
  s_{j,t} := \sum_{i=1}^{8993} S \big( x_{i,t}^{(j)} , \varphi_t (B_i) \big)
  \qquad \text{ and } \qquad \bar s_j := \frac{1}{5514} \sum_{t=1}^{5514} s_{j,t},
\end{align}
respectively, where $\varphi_t (B_i)$ is the observed number of events
in cell $B_i$ over the period from day $t$ to $t+6$. The mean score
$\bar s_j$ estimates the expected score of model $j$ and is thus a
measure of the relative forecast performance of this model.
Figure~\ref{fig:plot_pois_time} depicts the daily
scores~\eqref{eq:realizedScores} based on $S_\mathrm{pois}$ for the
four different models.  It uses a logarithmic scale, because the
values are much larger on days when events occur, in comparison to
days without events.  The FMC model consistently achieves the lowest
scores on days without earthquakes, since it consistently forecasts
the lowest number of events.  However, overall the LM model shows the
best performance in terms of mean scores over the whole testing period
\eqref{eq:realizedScores}, as can be seen in
Table~\ref{tab:performance_sum}.  This conclusion applies under both the
Poisson and the quadratic score.

\begin{table}[htb]
\caption{\label{tab:performance_sum}Summarized performance of the four
  models according to the mean score over the testing period $\bar s_j$
  from \eqref{eq:realizedScores}. The scoring functions used for
  evaluation are the Poisson (``pois'') and the quadratic (``quad'')
  score. Lowest values in each column are in boldface.}
\centering
\begin{tabular}{l cc}
\hline
Model   &  pois     & quad \\
\hline
LM  & \bf 2.68  & \bf 0.8218  \\
FMC &     2.76 &      0.8269  \\
LG  &     2.98 &      0.8275  \\
SMA &     2.70 &      0.8248 \\
\hline
\end{tabular}
\end{table}

To understand why the overall scores indicate superior predictive
ability of the LM model, we compute the mean score difference between
model $j$ and model $j'$ for each grid cell $i$ via
\begin{align}  \label{eq:score_diff}
  \Delta_i^{(j,j')} := \frac{1}{5514} \sum_{t=1}^{5514}
  \big( S_\mathrm{pois}(x_{i,t}^{(j)}, \varphi_t(B_i)) - S_\mathrm{pois}(x_{i,t}^{(j')}, \varphi_t(B_i)) \big) .
\end{align}
The left part of Figure~\ref{fig:spatial_comp} plots
$\Delta_i^{(1,2)}$, i.e.\ the mean score differences between the LM
and the FMC model per grid cell.  It illustrates that the lower
mean score of the LM model stems from its good performance in central
Italy in comparison to the FMC model.  The right part illustrates
aggregated performance, i.e.\ each pixel shows the performance when
the forecasts and observed values within a square neighbourhood
centred at this pixel are added up.  In this case the neighbourhood
has an edge length of 11 pixels.  Again, better predictive ability of
the LM model is most pronounced in central Italy and to a lesser
extent in the north, i.e.\ in areas where earthquake sequences
occurred during the study period.  The opposite is true for marine
regions around Sicily. 

\begin{figure}[bht]
\centering
\begin{minipage}{.49\textwidth}
	\includegraphics[clip, width = \textwidth, trim = 8 0 10 0]{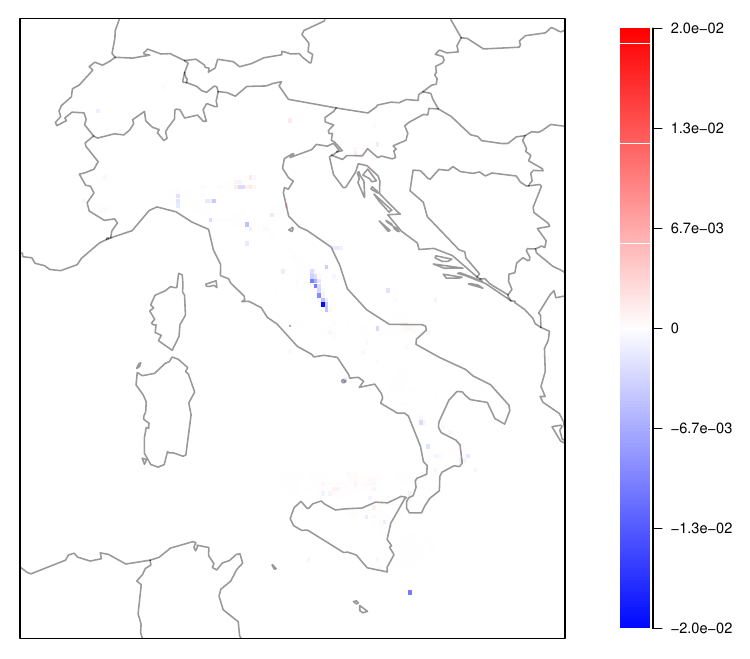}
\end{minipage}
\hfill
\begin{minipage}{.49\textwidth}
	\includegraphics[clip, width = \textwidth, trim = 8 0 10 0]{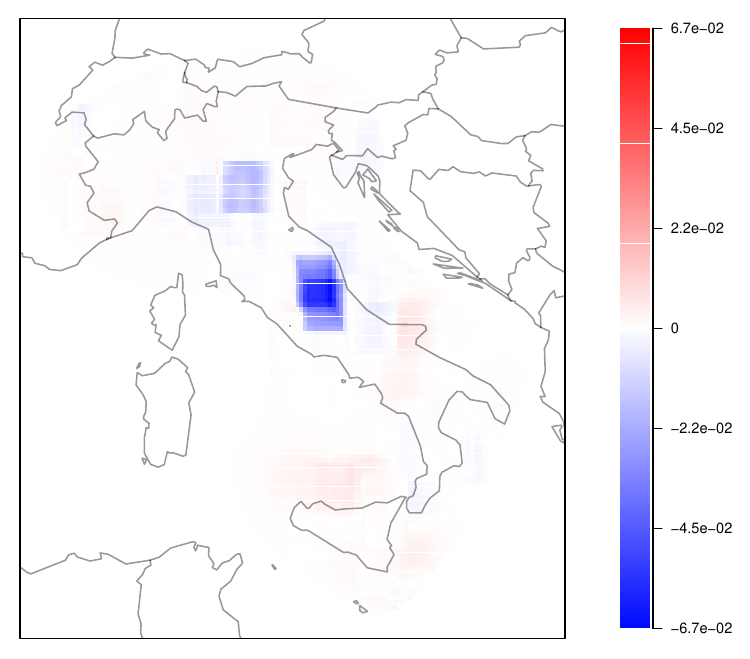}
	\centering
\end{minipage}
\caption{Mean score difference based on $S_\mathrm{pois}$ \eqref{eq:score_diff}
  between the LM and the FMC model, without (left) and with
  (right) aggregation.  Negative values (blue) indicate that the
  LM model has superior forecast performance, and positive
  values (red) vice versa.
  \label{fig:spatial_comp}
}
\end{figure}

Often, lack of data complicates the forecasting of point processes as
well as the proper testing of proposed forecasting models.  This
circumstance raises the question of how much data is needed
to reach valid conclusions on superior predictive ability.  As noted above, a
commonly used tool is the Diebold--Mariano test \citep{DiebMari1995},
which is a one-sample $t$-test applied to the score differentials, with adaptations to time series settings.
Standard power calculations for $t$-tests apply to independent samples,
where rules of thumb for the calculation of a required sample size or a
detectable difference are available \citep{Lehr1992, vanBelle2008}.  In
time series settings, rules of this type also require adaptation,
as exemplified in Section~S4 of the Supplementary Material, which
contains details on the analyses in this section.

\subsection{A new perspective on earthquake likelihood model testing}  \label{subsec:RELM}

An important element of the CSEP forecast experiments is a model
evaluation approach introduced by \citet{KaganJack1995} and
\citet{Schoretal2007}, to which we refer as earthquake likelihood
model testing (ELMT).  Further conceptual and computational
improvements are due to \citet{Zecharetal2010a},
\citet{Rhoadesetal2011}, and \citet{Ogataetal2013}.

Put simply, ELMT represents earthquakes by points in some region
$\cX \subset \real^k$, which is partitioned into grid cells
$B_1, \ldots, B_N$ for some $N \in \natural$, see e.g.\
Figure~\ref{fig:testing_region}.  The data consist of values
$x_1, \ldots, x_N \in \natural_0$ which count the earthquakes falling
in each cell.  A forecast or ``model'' is given by values
$\lambda_1, \ldots, \lambda_N \in (0, \infty)$ and its
``log-likelihood'' \citep{Schoretal2007} is defined as a sum of
Poisson log-likelihoods, i.e.\ via
\begin{align} \label{eq:binloglikelihood}
\ell (\lambda_1, \ldots, \lambda_N, x_1, \ldots, x_N) 
= \sum_{i=1}^{N}  \left( x_i \log \lambda_i - \log (x_i!) - 
\lambda_i \right) .
\end{align}
This terminology is motivated by the fact that, for a Poisson point
process with intensity measure $\Lambda$ such that
$\Lambda (B_i) = \lambda_i$, for $i = 1, \ldots, N$,
\eqref{eq:binloglikelihood} is the log-likelihood of the observation
$x_1, \ldots, x_N$.  Based on~\eqref{eq:binloglikelihood},
\citet{Schoretal2007} propose different tests.  Here we only consider
the test designed to compare forecasts.

The \textit{R-test}, or ratio test, compares two forecasts $A$ and $B$
specified by their grid cell values $\lambda_i^A$ and $\lambda_i^B$
for $i = 1, \ldots, N$, and aims to check whether model $A$ is at
least as good as model $B$.  The R-test considers the ``log-likelihood
ratio'' based on~\eqref{eq:binloglikelihood}, i.e.\
\begin{align} \label{eq:binloglikelihoodRatio}
R( A, B, x_1, \ldots, x_N) = \ell (\lambda_1^A, \ldots, \lambda_N^A,
x_1, \ldots, x_N) - \ell (\lambda_1^B, \ldots, \lambda_N^B, x_1, 
\ldots, x_N) ,
\end{align}
and then compares the realized value $z:= R( A, B , x_1, \ldots, x_N)$
to the distribution of the random variable
$Z := R( A, B , X_1, \ldots, X_N)$, where $X_1, \ldots, X_N$ are
independent Poisson random variables with parameters $\lambda_i^A$ for
$i=1, \ldots, N$. If $z$ lies in the lower tail of the distribution of
$Z$, then model $A$ is deemed worse than model $B$.  As the
distributional assumptions on $X_1, \ldots, X_N$ demonstrate, there is
an asymmetry inherent in the R-test: If model $A$ is tested against
model $B$, then the $X_i$ are assumed to have parameters $\lambda_i^A$
and if $B$ is tested against $A$, then $\lambda_i^B$ are assumed for
$X_i$.  As noted by \citet{Rhoadesetal2011} this implies that the
R-test is not really a comparative test, but rather a goodness-of-fit
test.  This explains seemingly contradictory results observed in
practice, where R-tests deem $A$ worse than $B$ and vice versa, see
also \citet{BraySchoen2013} for a discussion.  As a remedy,
\citet{Rhoadesetal2011} propose two modifications of the R-test, which
do not rely on a Poisson assumption to determine the distribution of
$Z$.

As pointed out by \citet{Harte2015}, ELMT suffers from several
drawbacks.  First, relying on a partition leads to a loss of
information, since the behaviour of models inside cells does not affect
the evaluation.  Moreover, assuming independence across cells as well
as a Poisson distribution leads to a likelihood mis-specification
under general point process models.  This prohibits the testing of
model characteristics other than cell expectations, since by reporting
$(\lambda_i)_{i=1,\ldots, N}$, every forecast is treated like a
Poisson point process.  However, as mentioned by
\citet{BraySchoen2013}, it is unclear how big the impact of the
Poisson assumption is on the testing results.

Taking the perspective of consistent scoring functions, we can answer
this question and clarify the role of the testing assumptions.  To
formalize ELMT in our setting, assume that the bounded domain $\cX$ is
partitioned into $k_n$ grid cells
$\cT_n = \{ B_1, \ldots, B_{k_n} \}$.  Based
on~\eqref{eq:binloglikelihood} and~\eqref{eq:binloglikelihoodRatio} we
define the \textit{cell scoring function}
$S_\mathrm{cell}^{\cT_n} : (0,\infty)^{k_n} \times \Mz \to \real$ via
\begin{equation}    \label{eq:score_bin}
S_\mathrm{cell}^{\cT_n} (\lambda_1, \ldots, \lambda_{k_n} , \varphi )
= \sum_{i=1}^{k_n} - \varphi (B_i) \log (\lambda_i) + \lambda_i
\end{equation}
for each partition $\cT_n$, $n \in \natural$. If $k_n = N$ and $x_i =
\varphi (B_i)$ for $i=1, \ldots, N$, then~\eqref{eq:binloglikelihoodRatio}
can be understood as the score difference between the forecasts
$\lambda_i^A$ and $\lambda_i^B$ with respect to $S_\mathrm{cell}^{\cT_n}$.
Since it applies the scoring function~\eqref{eq:PoisScore} to each grid 
cell, $S_\mathrm{cell}^{\cT_n}$ is strictly
consistent for the collection of cell expectations $\E \Phi (B_i)$,
$B_i \in \cT_n$, cf.~Example~\ref{ex:SimpleProperty}.  This shows that
the Poisson log-likelihood in~\eqref{eq:binloglikelihoodRatio} can be
used for a sound comparison of cell expectations, since the true
expectations obtain the minimal expected score.  We emphasize that
this conclusion holds regardless of whether or not the data or the
forecasts are based on Poisson point processes.  Moreover, dependence
among cells is irrelevant for this fact, since (strict) consistency
concerns only \textit{expected} scores.  Hence, the validity of
statistical methods which rely on the expected scores
of~\eqref{eq:score_bin} is not limited to Poisson models nor to
Poisson point process data.  In a nutshell, these methods assess
forecast performance in terms of cell expectations only, since the
scoring function~\eqref{eq:PoisScore} is strictly consistent for the
expectation.  For instance, the symmetric modifications of the R-test
due to \citet{Rhoadesetal2011} can be seen as Diebold--Mariano~(DM)
tests \citep{DiebMari1995} based on $S_\mathrm{cell}^{\cT_n}$.  Hence,
they test whether one model is better than its competitor in
forecasting the mean number of earthquakes in the cells.  Note that
although such methods are valid for arbitrary point processes,
considerable spatial or temporal dependencies will affect significance
levels and deteriorate their ability to detect differences in forecast
performance in finite samples.
 
It remains to discuss the role of the partitioning of $\cX$ into grid
cells.  To understand its implications, note that just as the Poisson
distribution leads to the scoring function~\eqref{eq:PoisScore} for
the expectation, the Poisson point process can be used to obtain a
scoring function for the intensity (Section~\ref{subsec:Intensity}).
The reason is that every intensity report induces a Poisson point
process with this intensity and these processes can then be compared
via the logarithmic score~\eqref{eq:PPlogScore}, which attains the
value~\eqref{eq:Poisson_logScore} for Poisson densities.  In the
setting of Section~\ref{subsec:Intensity}, we can formalize as
follows.

\begin{prop} \label{prop:PoisScoreII} Let every element of $\Mf$ admit
  a density $\lambda$ with respect to Lebesgue measure.  Then the
  scoring function $S: \Mf \times \Mz \to \real$ defined by
\begin{align}   \label{eq:intens_PoisScore}
  S( \Lambda, \{y_1, \ldots, y_n\} ) =
  - \sum_{i=1}^{n} \log \lambda(y_i) + \int_\cX \lambda(y) \dd y
\end{align}
for $n \in \natural$, and
$S(\Lambda, \emptyset) = \int_\cX \lambda (y) \dd y$, is a strictly
consistent scoring function for the intensity.
\end{prop}

\begin{proof}
  The scoring function~\eqref{eq:intens_PoisScore} corresponds to $S$
  from Proposition~\ref{prop:NormalizedIntens} when choosing the
  logarithmic score for $S'$, the Bregman
  function~\eqref{eq:PoisScore} for $b$ and $c=1$. Since $S'$ is
  strictly consistent and $b$ is strict, $S$ is strictly consistent
  for the intensity. 
\end{proof}

The scoring function~\eqref{eq:intens_PoisScore} can be interpreted as
a point process analogon to the Dawid-Sebastiani-score
\citep{DawidSebas1999}. While the Dawid-Sebastiani-score relies on the
first and second moments of the predictive distribution, this scoring
function depends on the intensity only.

The next result shows that the cell scoring function
$S_\mathrm{cell}^{\cT_n}$ serves as an approximation to the scoring
function~\eqref{eq:intens_PoisScore}.  Essentially, if a forecaster
does not report an intensity $\lambda$, but only the integrals
$\lambda_i^{(n)}$ of $\lambda$ over the collection of grid cells
${\cT_n}$, then forecast comparison using the cell scoring function
$S_\mathrm{cell}^{\cT_n}$ is on par with a comparison based on the
scoring function~\eqref{eq:intens_PoisScore}, provided the partition
is sufficiently fine.  The correction term in \eqref{eq:approx_conv}
does not affect the evaluation, as it is independent of the reported
integrals.  To make this precise, we follow \citet{DVJvol1} and call a
sequence of partitions $(\cT_n)_{n \in \nat}$ \textit{dissecting} if
it is nesting and asymptotically separates every pair of points.

\begin{prop} \label{prop:approx_spatial} Let
  $\lambda : \cX \to (0, \infty)$ be an intensity and
  $(\cT_n)_{n \in \nat}$ a dissecting system of measurable partitions
  of $\cX$ which generates the Borel $\sigma$-algebra on $\cX$. Let
  $P_0 \in \cP$ be the distribution of the unit rate Poisson point
  process on $\cX$ and define partition integrals
\begin{align*}
\lambda_i^{(n)} = \int_{B_i^{(n)}} \lambda(y) \dd y ,
\end{align*}
for all $i = 1, \ldots, k_n$, $B_i^{(n)} \in \cT_n$, and $n \in \nat$.
Then 
\begin{align}   \label{eq:approx_conv}
S_\mathrm{cell}^{\cT_n} \big( \lambda_1^{(n)}, \ldots, 
\lambda_{k_n}^{(n)}, \varphi \big) + \sum_{i=1}^{k_n} \one{ 
\varphi(B_i^{(n)}) > 0} \log ( \vert B_i^{(n)} \vert) 
\longrightarrow S(\Lambda, \varphi),
\end{align}
for $P_0$-a.e.\ $\varphi \in \Mz$ as $n \to \infty$, where $S$ is the
scoring function~\eqref{eq:intens_PoisScore}.
\end{prop}

\begin{proof}
  Let $\varphi = \{y_1, \ldots, y_m \}$ with $m \in \nat_0$ be a point
  process realization.  For a large $n \in \nat$ every set $B_i^{(n)}$
  contains at most one point of $\varphi$, and we let $i_n(j)$ denote
  the index of the set such that $ y_j \in B_{i_n(j)}^{(n)}$ for
  $j = 1, \ldots, m$.  Then the left-hand side
  of~\eqref{eq:approx_conv} equals
\begin{align*}
& \phantom{=} - \sum_{i=1}^{k_n} \left( \varphi (B_i^{(n)}) \log \Big( 
\int_{B_i^{(n)}} \lambda(y) \dd y \Big) - \one{ \varphi(B_i^{(n)})
> 0} \log ( \vert B_i^{(n)} \vert) - \int_{B_i^{(n)}} \lambda(y) 
\dd y \right) \\
& = - \sum_{j=1}^{m} \log \left( \vert B_{i_n(j)}^{(n)} \vert^{-1}
\int_{B_{i_n(j)}^{(n)}} \lambda(y) \dd y \right)  + \int_\cX \lambda(y) \dd y \\
& \longrightarrow - \sum_{j=1}^{m} \log(\lambda (y_j)) + \int_\cX
\lambda(y) \dd y
\end{align*}
for $n \to \infty$ and $P_0$-a.e.\ $\varphi \in \Mz$.  The last line
follows from an approximation result for the Radon-Nikod\'ym
derivative $\lambda$ \citep[Lemma~A1.6.III]{DVJvol1}. 
\end{proof}

Propositions~\ref{prop:PoisScoreII} and~\ref{prop:approx_spatial} show
that comparisons based on the Poisson
log-likeli\-hood~\eqref{eq:binloglikelihoodRatio} can be understood as
approximations to a comparison of intensity forecasts with the scoring
function~\eqref{eq:intens_PoisScore}.  In particular, we can conclude
that partitioning is not essential for model evaluation: A
straightforward generalization of ELMT relies on models that produce
intensities $\lambda : \cX \to (0, \infty)$ on the testing region,
which can then be compared via consistent scoring functions
(Section~\ref{subsec:Intensity}), with~\eqref{eq:intens_PoisScore}
giving one possible choice.  However, in some situations partitioning
might be desirable, e.g.\ when no explicit expression for the
intensity is available. This also applies to our case study, where
only the expected numbers per grid cell were produced by the
forecasting models.  In light of
Proposition~\ref{prop:approx_spatial}, our evaluation is essentially a
comparison of the point process intensities forecasted by the four
competing models.

\section{Discussion}  \label{sec:discussion}

Assessing forecast accuracy and comparing the performance of several
competing forecasts is a non-trivial task that poses challenges across
disciplines and sectors.  In this paper we have demonstrated that
consistent scoring functions allow for the comparative evaluation of
point process forecasts.  Our methods are complementary to the
simulation-based approach of \citet{Heinetal2019}, encompass existing
techniques for model comparison, and yield a novel understanding of
earthquake likelihood model testing.  In particular, we have shown
that the Poisson log-likelihood can be used for theoretically
principled comparative forecast evaluation in terms of cell
expectations.  This is an important finding, as it supports current
practice in comparisons between Poisson models, for which the
interpretation in terms of log-likelihood is useful and welcome, and
other types of models, which might generate cell expectations only.
When one ignores the possibility of multiple events in a cell, the
cell expectation equals the probability of an event, and we are in
the setting studied by \citet{Serafetal2022}.

To conclude our study, we continue the discussion of methods for model
comparison that are based on the log-likelihood, i.e.\ the model
log-density evaluated at the observations, distinguish relative and
absolute performance assessment, and hint at future work.

The \textit{entropy score} considers the log-likelihood of probability
forecasts induced by a point process model \citep{DaleyVe-Jo2004,
  HarteVe-Jo2005}.  It can be interpreted as an application of the
logarithmic score to probabilistic predictions in terms of numbers of
events.  The expected value of the entropy score difference between a
model of interest and a reference model yields the \textit{information
  gain}.  \citet{DaleyVe-Jo2004} note that the information gain is an
inherent characteristic of a point process model that quantifies
predictability and relates closely to entropy.  A detailed discussion
of the relationships between proper scoring rules, entropy, and
divergences is available in Section 2.2 of \citet{GneitRaft2007}.

Information criteria such as AIC or BIC assess the relative quality of
competing models, and can be applied to point process models, provided
that densities are available, see e.g.~\citet{Chenetal2018}.  They
connect naturally to consistent scoring functions through their
goodness-of-fit component, which usually consists of a log-likelihood
and thus finds the logarithmic score~\eqref{eq:Pseudo_and_Log} for the
model at hand.  The penalty component, which depends on the number of
fitted parameters, is a necessary correction when operating in-sample,
i.e.\ relying on the same data as used for model fitting.  In
contrast, comparative forecast evaluation via scoring functions is
tailored to out-of-sample settings, as in our case study.

In Bayesian settings, a standard approach to model comparison is the
use of \textit{Bayes factors} of a model vs.~a competitor, as employed
by \citet{Marzetal2012} in earthquake likelihood model testing.
Similar to information criteria, Bayes factors are closely connected
to the logarithmic score \citep[Section~7]{GneitRaft2007}.

A further likelihood-based method for point processes uses deviance
residuals, as proposed by \citet{Clementsetal2011}.  In general, point
process residuals form an empirical process arising from fitting a
conditional intensity to data \citep{Schoen2003, Baddeleyetal2005, 
  Brayetal2014}.  Residuals can be used to assess goodness-of-fit and
especially indicate in which regions a model fits well or poorly.
\citet{Clementsetal2011} propose a graphic comparison of models for
the conditional intensity by plotting the log-likelihood ratio across
a partition of the spatial domain, which can be interpreted as
visualizing local differences in the logarithmic score.

Consistent scoring functions, as well as the just discussed methods,
compare competing models or forecasts.  This contrasts with many
existing point process model evaluation tools, which focus on absolute
performance, e.g.\ based on calibration \citep{Thorarinsdottir2013}
and goodness-of-fit.  Although this is important in model building, a
selection among the available competitors has to be done eventually,
and measures of absolute performance are not designed, and hence tend
to be poorly positioned, for this task.  Moreover, as pointed out by
\citet{NoldeZieg2017}, focusing on absolute performance may lead to
misguided incentives in designing candidate models.

Earthquake likelihood model testing, a central element of the CSEP
forecasting experiments, is tacitly based on strictly consistent
scoring functions for expectations.  A principled use of these
functions, as illustrated in our case study, provides valid
comparisons of forecasted intensities.  Importantly, common
assumptions in the context of CSEP tests are not needed for such an
evaluation: Neither the forecasting models, nor the data, need to
follow any Poisson or independence assumption, and with suitably
adapted models, partitioning the testing region can be avoided.  As
these conclusions apply to intensity forecasts, a natural next step is
to employ consistent scoring functions to compare earthquake forecasts
in terms of other statistical properties.  In particular, dependence
properties or full distributions are natural candidates for forecast
evaluation in the CSEP framework \citep{Schoretal2018,
  Nandanetal2019}.  The choice and implementation of consistent
scoring functions in settings of this type pose challenges for future
work.  

{\footnotesize
\paragraph{Acknowledgements}
  Jonas Brehmer and Tilmann Gneiting are grateful for support by the
  Klaus Tschira Foundation. Jonas Brehmer gratefully acknowledges support
  by the German Research Foundation (DFG) through Research
  Training Group RTG 1953.  Part of this research came to fruition
  during mutual visits of Kirstin Strokorb at the University of
  Mannheim and Jonas Brehmer and Martin Schlather at Cardiff
  University during a workshop funded by the London Mathematical
  Society.  We thank our hosting institutions for their generous
  hospitality.  The authors would also like to thank Claudio
  Heinrich-Mertsching, Christopher D\"orr and Alexander Jordan for
  helpful discussions, and Kristof Kraus for code review.  
  Likewise, we are grateful to the anonymous reviewers for their
  comments that helped improve the clarity of this paper.}

\section*{Supplementary Material}

The Supplementary Material contains additional technical details and 
further simulation experiments.  \textsf{R} code for reproduction is publicly available \citep{Brehmer2023}.  Data are available from the authors upon request.

{\small
\bibsep=0pt

}

\clearpage

\setcounter{page}{1}
\setcounter{table}{0}
\setcounter{section}{0}

\setcounter{equation}{0}

\setcounter{example}{0}
\setcounter{prop}{0}
\setcounter{cor}{0}
\setcounter{lemma}{0}
\setcounter{theorem}{0}

\renewcommand{\thepage}{\roman{page}}

\renewcommand\thesection{S\arabic{section}}
\renewcommand{\theHsection}{S\arabic{section}}
\renewcommand{\theequation}{S\arabic{equation}}
\renewcommand\thetable{S\arabic{table}}
\renewcommand\thefigure{S\arabic{figure}}

\renewcommand\theexample{S\arabic{example}}
\renewcommand\thelemma{S\arabic{lemma}}
\renewcommand\thetheorem{S\arabic{theorem}}
\renewcommand\thecor{S\arabic{cor}}
\renewcommand\theprop{S\arabic{prop}}

\newcommand{\newindex}{\Large Contents}
\newlistof{ind}{tce}{\newindex}
\newcommand\newsection[1]{%
  \phantomsection
  \addcontentsline{tce}{section}{\protect\makebox[1.3em][l]{\thesection}#1}}
\newcommand\newsubsection[1]{%
  \phantomsection
  \addcontentsline{tce}{subsection}{\protect\makebox[2.1em][l]{\thesubsection}#1}}

\begin{center}
\bf \Large 
Supplement to: \\
Comparative evaluation of point process forecasts
\end{center}
\bigskip

\listofind

\section{Discussion of point process scenarios}
\newsection{Discussion of point process scenarios}
\label{sec:scenarios}

This section extends the discussion at the end of Section~2.

In the main manuscript we focus on the setting where a spatial point
process $\Phi$ on some domain $\cX \subset \real^d$ is observed at
fixed points in time.  For example, the case study (Section~5)
considers daily observations of locations of earthquakes in Italy.
However, forecasting for point processes appears in a variety of other
situations, and the use of strictly consistent scoring functions
adapts readily.  To clarify this idea, we distinguish three different
point process scenarios.  Although motivated by commonly encountered
applications, there might be settings where the distinction is
artificial.

\begin{description}
\item{\bf Scenario A} (purely spatial) In this scenario, the process
  is defined on either a single spatial domain (Scenario A1), or
  several non-overlapping subdomains (Scenario A2).  Examples include
  points fixated by observers of images \citep{Bartetal2013} and
  locations of trees in a forest \citep{StoyPent2000}.  Stationarity
  is a common simplifying assumption in this context.

\item{\bf Scenario B} (purely temporal) In this scenario, there is no
  spatial component and the process concerns points in time only.
  Examples are arrival times of e-mails \citep{Foxetal2016} and times
  of infection with a disease \citep{Schoenbetal2019}. In this special
  setting the directional character of time allows for a distinct
  interpretation and treatment.

\item{\bf Scenario C} (spatio-temporal) In addition to the spatial
  component, processes in this scenario possess a temporal component,
  which could be discrete (Scenario C1) or continuous (Scenario C2).
  Examples include locations and times of crimes in a city
  \citep{Mohleretal2011} and earthquakes observed over time in a
  specific region \citep{Ogata1998,Zhuangetal2002}.  The main
  manuscript focuses on Scenario~C1.
\end{description}

In order to compare forecasts in each of these scenarios, we can in
principle proceed as in Sections~4 and 5: Choose a strictly consistent
scoring function $S$ for a statistical property of point processes,
e.g.\ the intensity, and find the mean score difference
\begin{align*}
\frac 1n \sum_{i=1}^n \left( S(r_i, \varphi_i) - S(r_i^*, \varphi_i) \right)
\end{align*}
for forecast reports $r_i$ and $r_i^*$ and associated observed point
patterns $\varphi_i$, where the index $i = 1, \ldots, n$ represents
repeated observations.  Then negative values support forecast $r$,
while positive values support $r^*$.  The mean score difference is an
estimator of the expected score difference
$\E \left( S(r, \Phi) - S(r^*, \Phi) \right)$, and implementation
details vary across scenarios, also impacting the assessment of the
uncertainty inherent in the estimate, which is of particular
importance when tests for superior predictive performance are sought.
To illustrate the key ideas we distinguish whether the point process
has a continuous or discrete time component.

\paragraph{Discrete time}
 
Assume that the point process is sampled at fixed points in time,
i.e.\ it can be modelled by a sequence $(\Phi_t)_{t \in \natural}$
adapted to a filtration $(\cH_t)_{t \in \natural}$.  This setting
includes the special case of i.i.d.\ realizations and relates to
Scenario~C1 as well as variants of Scenario~A with repeated
observations.  Given two forecast sequences $(R_t)_{t \in \natural}$
and $(R_t^*)_{t \in \natural}$ the score differences
$(S (R_t, \Phi_t) - S(R_t^*, \Phi_t))_{t \in \natural}$ form a
sequence of real-valued random variables, thus the common
Diebold--Mariano (DM) tests \citep{DiebMari1995} are directly
applicable.  We briefly discuss the more general forecast comparison
framework of \citet{NoldeZieg2017} in our setting.  Let $S$ be
strictly consistent for a point process statistic
$\Gamma : \cP \to \sA$ and assume that forecasts in terms of $\Gamma$
applied to the conditional distribution $\Phi_t \mid \cH_{t-1}$ are
given.  These forecasts can be regarded as random sequences
$R = (R_t)_{t \in \natural}$ and $R^* = (R_t^*)_{t \in \natural}$ such
that $R_t$ and $R^*_t$ are $\cH_{t-1}$-measurable.  Their forecast
performance can be compared via the \textit{mean score difference}
\begin{equation}  \label{eq:averageScores}
\Delta_n (R, R^*) := \frac{1}{n} \sum_{t=1}^{n} S( R_t, \Phi_t) - 
\frac{1}{n} \sum_{t=1}^{n} S(R_t^* , \Phi_t) 
= \frac 1n \sum_{t=1}^{n} \left( S( R_t , \Phi_t) - 
S(R_t^*, \Phi_t) \right) ,
\end{equation}
which is an estimator for the difference in expected scores.  Based on
the law of large numbers and the strict consistency of $S$, a positive
value supports the hypothesis that $R^*$ is superior to $R$, while a
negative value supports the opposite hypothesis.  A further step is to
test whether $\Delta_n (R, R^*)$ is significantly different from zero.
In the simple situation of an i.i.d.\ sequence
$(\Phi_t)_{t \in \nat}$, the forecast sequences reduce to
$r, r^* \in \sA$, i.e.\ they are constant in time.  We can then test
for significant differences in expected scores based on the asymptotic
normality of the well-known $t$-statistic
$t_n := \sqrt{n} \Delta_n (r, r^*) / \sqrt{\hat \sigma_n^2} $, where
$\hat \sigma_n^2$ estimates the variance of
$ S(r,\Phi) - S(r^*, \Phi)$.  For dependent time series
$(\Phi_t)_{t \in \natural}$, $(R_t)_{t \in \natural}$, and
$ (R_t^*)_{t \in \natural}$ we refer to \citet{NoldeZieg2017}, where
tests for equal forecast performance rely on suitable asymptotic
results developed in \citet{GiacWhite2006}.

\paragraph{Continuous time}

If we consider point processes in Scenario~C2 or Scenario~B, then
temporal dependence between the points of $\Phi$ becomes an essential
feature of the process and can also be object of the forecast.  For
instance, the statistic $\Gamma$ might consist of temporal features of
the point process.  Also, dependencies need to be accounted for in
estimation and testing, as they affect asymptotic distributions.  To
illustrate this, assume for simplicity that $\Phi$ is a purely
temporal process observed over a time period $[0,T]$ with
$0 <t_1 < \cdots < t_k <T$ denoting the corresponding arrival times.
Moreover, let $R_i$ and $R_i^*$ be reports issued at time $t_{i-1}$
based on the previous arrivals $t_1, \ldots, t_{i-1}$.  This yields a
realized score difference
\begin{align}   \label{eq:averageScoresTime}
\Delta_T (R, R^*) &= \sum_{i=1}^{n(T)} \left( S ( R_i , t_i ) 
- S ( R_i^* , t_i ) \right),
\end{align}
where $n(T) := \Phi ((0,T])$ is the random number of points in
$[0,T]$.  In contrast to~\eqref{eq:averageScores} we do not consider
averages since $n(T)$ is a random variable depending on $\Phi$ and
dividing by it will interfere with the consistency of $S$.  The score
difference $\Delta_T (R, R^*)$ is a sum of a random number of random
variables, usually called a random sum.  This perspective connects the
estimation of score differences to the theory of total claim amount in
insurance, see e.g.\ \citet{Mikosch2009} and \citet{EKM1997}.

Asymptotic results for the score
difference~\eqref{eq:averageScoresTime} for $T \to \infty$ are
desirable to assess how uncertainty affects forecast evaluation and
transfer the DM test to the continuous time setting.  One possible
approach to this problem relies on limit theorems for randomly indexed
processes due to \citet{Anscombe1952}, in particular random central
limit theorems: If the number of points $n(T)$ satisfies a weak law of
large numbers, then under Anscombe's condition, we only need to ensure
that the sequence
$(S ( R_i , t_i ) - S ( R_i^* , t_i ) )_{i \in \natural}$ satisfies a
central limit theorem in order to obtain asymptotic normality
for~\eqref{eq:averageScoresTime}.  Such results are available for
strong mixing \citep{Lee1997}, $\psi$-weakly dependent
\citep{HwangShin2012}, and $m$-dependent \citep{Shang2012} sequences.
Working these into tests for superior forecast performance for
(spatio-)temporal point processes is an avenue for future work.

\section{Further scoring functions for point processes}
\newsection{Further scoring functions for point processes}

The technical context of this section is the same as in Section~3.

\subsection{Simple examples}
\newsubsection{Simple examples}

The subsequent examples are applications of the transformation
principle (Proposition~1).

\begin{example}[void probability]
  For any fixed set $B \in \cB (\cX)$ the functional $\Gamma$ defined
  via
  $\Gamma (P) = P ( \{ \varphi \mid \varphi \cap B = \emptyset \} )$
  is elicitable.  This follows from Proposition~1 with
  $T(F) = \E_F Y$ and $g(\varphi) = \one{ \varphi (B) = 0}$.  Strictly
  consistent scoring functions for $\Gamma$ are of the Bregman
  form~(2), see also Example~1.
\end{example}

\begin{example}[point process integrals]  
\label{ex:RevelationPPcorI}
Fix measurable functions $f_i : \cX \to \real$, $i = 1, \ldots, m$ for
$m \in \natural$. Define $g: \Mz \to \real^m$ via
\begin{equation*}
g(\varphi) = \left( \int_\cX f_1 \dd \varphi, \ldots, \int_\cX f_m \dd \varphi
\right)^\top = \left( \sum_{x_i \in \varphi} f_1 (x_i) , \ldots, 
\sum_{x_i \in \varphi} f_m (x_i) \right)^\top  ,
\end{equation*}
set $g(\cP):= \lbrace P \circ g^{-1} \mid P \in \cP \rbrace$ and let
$T = \mathrm{id}_{g(\cP)}$.  Then the finite-dimensional distribution
functional $\Gamma_{f_1, \ldots, f_m} (P) = T(P \circ g^{-1})$ is an
elicitable property of the point process $\Phi$.  Consistent scoring
functions for $\Gamma$ are obtained by applying consistent scoring
functions for distributions \citep{GneitRaft2007} to the $m$-variate
distribution $P \circ g^{-1}$, see also \citet{Heinetal2019}.
\end{example}

\subsection{Distribution and density}
\newsubsection{Distribution and density}

This material extends Section~3.2.

\paragraph{General result for the full distribution}

The law $P_\Phi$ of a finite point process on $\cX$ can be
equivalently represented by two sequences $(p_k)_{k \in \natural_0}$
and $(\Pi_k)_{k \in \natural}$.  Each $p_k$ specifies the probability
of finding $k$ points in a realization.  The $\Pi_k$ are symmetric
probability measures on $\cX^k$ which describe the distribution of any
ordering of points, given $k$ points are realized, see
\citet[Chapter~5.3]{DVJvol1} for details.

To state the next result, we introduce the notion of
\textit{symmetric} scoring functions, where
$S : \sA \times \real^n \to \real$ is called symmetric if
$S(a, y_1, \ldots, y_n) = S(a, y_{\pi (1)}, \ldots, y_{\pi(n)})$ for
all $a \in \sA$, $y \in \real^n$ and permutations $\pi$. Symmetry
ensures that the scoring functions in the subsequent proposition are
independent of the enumeration of the realization of $\Phi$.

\begin{prop} \label{prop:PPdistribution} Let $\cP$ be a class of
  distributions of finite point processes, with $Q \in \cP$ decomposed
  into $(\Pi_k^Q)_{k \in \natural}$ and $(p_k^Q)_{k \in
    \natural_0}$. Set $\cF_k := \{\Pi_k^Q \mid Q \in \cP \}$ and let
  $S_k : \cF_k \times \cX^k \rightarrow \real$ be a symmetric
  consistent scoring function for $\mathrm{id}_{\cF_k}$ for all
  $k \in \natural$.  Let $S_0$ be a consistent scoring function for
  distributions on $\natural_0$.  Then the function
  $S: \cP \times \Mz \to \real$ defined via
\begin{align*}
S( ((\Pi_k^Q)_{k \in \natural} , (p_k^Q)_{k \in \natural_0}), 
\{y_1, \ldots, y_n\} ) = S_n (\Pi_n^Q, y_1, \ldots , y_n) + 
S_0 ( (p_k^Q)_{k \in \natural_0} , n)
\end{align*}
for $n \in \natural$ and $S( ((\Pi_k^Q)_{k \in \natural} , 
(p_k^Q)_{k \in \natural_0}), \emptyset ) := S_0 ( (p_k^Q)_{k \in
\natural_0} , 0)$ is a consistent scoring function for the
distribution of the point process $\Phi$.  It is strictly consistent
if $S_0$ and $(S_k)_{k \in \natural}$ are strictly consistent.
\end{prop}

\begin{proof}
The result follows by decomposing the expectation $\E_P S(Q, \Phi)$
into expectations on the sets $\{\Phi = n \}$ for $n \in \natural$ and
using the (strict) consistency of $S_n$ on each set.
\end{proof}

\paragraph{Hyv\"arinen score}

Assume that a point process model admits explicit expressions for the
Janossy densities $(j_k)_{k \in \natural_0}$ (see Section~3.2),
however, only up to an unknown normalizing constant.  In this
situation, 0-homogeneous consistent scoring functions for densities
can be of use, as they allow for the consistent evaluation of an
unnormalized density. The most relevant example is the
\textit{Hyv\"arinen score} defined via
\begin{align*}
\mathrm{HyvS}(f,y) := \Delta \log f(y) + \frac 12 \Vert \nabla \log f(y) \Vert^2 ,
\end{align*}
where $\nabla$ denotes the gradient, $\Delta$ is the Laplace operator,
and $f$ is a twice differentiable density on $\real^d$. To ensure
strict consistency on a class of probability densities $\cL$ its
members have to be positive almost everywhere and for all $f , g \in
\cL$ it must hold that $\nabla \log( f(y)) g(y) \to 0$ as $\Vert y 
\Vert \to \infty$, see \citet{Hyvae2005}, \citet{Parryetal2012}, and
\citet{EhmGneit2012} for details.

Similar to the logarithmic score, we can transfer the Hyv\"arinen
score to the point process setting.  To do this we assume that for all
$Q \in \cP$ and $k \in \natural$, $j_k^Q$ is defined on $(\real^d)^k$
and satisfies the aforementioned regularity conditions.  Then the
function $S: \cP \times \Mz \to \real$ defined via
\begin{align}   \label{eq:PPHyvScore}
S( (j_k^Q)_{k \in \natural_0}, \{y_1, \ldots, y_n\} ) 
=  \mathrm{HyvS} (j_n^Q , y_1, \ldots , y_n )
\end{align}
for $n \in \natural$ and $S( (j_k^Q)_{k \in \natural_0}, \emptyset ) 
:= 0$ is a consistent scoring function for the distribution of the
point process $\Phi$.  Observe that we cannot achieve strict
consistency for $S$, since the probability of $\vert \Phi \vert = n$
is proportional to $j_n$ and thus not accessible to the Hyv\"arinen
score.

\begin{example}[Gibbs point process]   \label{ex:GibbsLikelihood}
Stemming from theoretical physics,  Gibbs processes are a popular tool
to model particle interactions.  They are defined via their Janossy
densities
\begin{equation*}
j_n (y_1 , \ldots, y_n) = C(\theta) \exp \left( - \theta U(y_1, \ldots, y_n)  \right) ,
\end{equation*}
where $U$ represents point interactions, $\theta$ is a parameter often
referred to as temperature, and $C$ is the partition function, which
ensures that the collection $(j_k)_{k \in \natural_0}$ is properly
normalized, see e.g.\ \citet[Chapter~5.3]{DVJvol1} and
\citet[Chapter~5.5]{Chiuetal2013}.  It is in general difficult to find
closed form expressions for $C$, or even to approximate it, hence the
Hyv\"arinen score might seem attractive to evaluate models based on
$(j_k)_{k \in \natural_0}$.  Plugging $j_n$ into~\eqref{eq:PPHyvScore}
gives
\begin{align*}
S((j_k)_{k \in \natural_0} , \{y_1, \ldots, y_n \} ) = \theta 
\left( - \Delta U(y_1, \ldots, y_n) + \frac{\theta}{2} \Vert \nabla 
U(y_1, \ldots, y_n) \Vert^2 \right)
\end{align*}
for $n \in \natural$, where the derivatives are computed with respect
to the coordinates of the vector $(y_1, \ldots, y_n) \in (\real^d)^n$.
The simplest choice for interactions is to restrict $U$ to first- and
second-order terms
\begin{align*}
U (y_1, \ldots, y_n) := \sum_{i=1}^{n} l(y_i) + 
\sum_{i,j = 1}^{n} \psi \left( \Vert y_i - y_j \Vert^2 \right)
\end{align*}
for $l : \real^d \to \real$ and $\psi : [0, \infty) \to [0, \infty)$
with $\psi (0) = 0$, see e.g.\ \citet[Chapter~5.3]{DVJvol1}.  To apply
the Hyv\"arinen score in this setting, $l$ and $\psi$ have to satisfy
regularity conditions detailed above and in \citet{Hyvae2005}, and in
particular admit second order derivatives almost everywhere.  The
soft-core models for $\psi$ introduced in \citet{OgatTane1984} satisfy
this condition, while their hard-core model for $\psi$ is not even
continuous.  An additional technical issue is that
\citet{OgatTane1984} consider point processes on a finite domain $\cX$
and use a constant $l$.  To make the Hyv\"arinen score applicable in
this setting a possible solution is to approximate their models via
twice differentiable densities on $(\real^d)^n$.
\end{example}

\subsection{Moment measures}
\newsubsection{Moment measures}
\label{subsec:moment_measure}

Moment measures can be interpreted as the point process analogue to
the moments of a univariate random variable.  Strictly consistent
scoring functions for these measures can be constructed in the same
way as for the intensity, see Proposition~3.4.

For $n \in \natural$, let $\Mf^n = \Mf( \cX^n)$ be the set of finite
Borel measures on $\cX^n$.  For positive measurable functions $f: \cX^n
\to (0, \infty)$ the \textit{$n$-th moment measure} $\mu^{(n)}$ and
the \textit{$n$-th factorial moment measure} $\alpha^{(n)}$ are
defined via the relations
\begin{align*}
\mathbb{E} \left( \underset{x_1, \ldots, x_n \in \Phi}{\sum}
f(x_1, \ldots, x_n) \right) &= \int_{\cX^n} f(x_1, \ldots, x_n) \dd
                              \mu^{(n)} (x_1, \ldots, x_n) ,
\end{align*}                              
and
\begin{align*}
\mathbb{E} \left( \underset{x_1, \ldots, x_n 
\in \Phi}{\sum\nolimits^{\neq}} f(x_1, \ldots, x_n) \right)
&= \int_{\cX^n} f(x_1, \ldots, x_n) \dd \alpha^{(n)} 
(x_1, \ldots, x_n) ,
\end{align*}
respectively, see e.g.\ \citet{Chiuetal2013} and  \citet{DVJvol1}.
Here $\Sigma^{\neq}$ denotes summation over all $n$-tuples that
contain distinct points of $\Phi$.  Using the notion of
\textit{factorial product} defined via
\begin{align*}
m^{[n]} := \left\lbrace 
\begin{array}{lr}
m (m-1) (m-2) \cdots (m-n+1) &, \, m \geq n \\
0 &, \, m <n
\end{array}
\right.
\end{align*}
for $m,n \in \natural$ we obtain the concise representations
$\mu^{(n)} (B^n) = \E \Phi (B)^n$ and $\alpha^{(n)} (B^n) = \E \Phi
(B)^{[n]}$ for Borel sets $B \in \mathcal{B} (\cX)$, see e.g.\
\citet[Chapter~5]{DVJvol1}.

\begin{prop} \label{prop:MomentM} Set
  $\cF^n := \{ P^* \mid P \in \Mf^n \}$, let
  $S: \cF^n \times \cX^n \rightarrow \mathbb{R}$ be a consistent
  scoring function for $\mathrm{id}_{\cF^n}$ and
  $b : [0, \infty) \times [0, \infty) \rightarrow \real$ a Bregman
  function.
\begin{enumerate}[label=(\roman*)]
	\item The function $S_1: \Mf^n \times \Mz \to \real$ defined via
	\begin{align*}
	S_1 (\mu, \{y_1, \ldots, y_m \} ) =  \underset{x_1, \ldots, x_n 
	\in  \{y_1, \ldots, y_m \}}{\sum} S(\mu^* , x_1, \ldots, x_n) 
	+ c \hsp b(\mu(\cX^n), m^n )
	\end{align*}
	for $m \in \natural$, and $S_1 (\mu, \emptyset) = c \hsp b(\mu
        (\cX^n), 0)$ for $c > 0$, is a consistent scoring function for
        the $n$-th moment measure.
	\item The function $S_2: \Mf^n \times \Mz \to \real$ defined via
	\begin{align*}
	S_2 (\alpha,  \{y_1, \ldots, y_m \}) = \underset{x_1, \ldots, x_n
	\in  \{y_1, \ldots, y_m \}}{\sum\nolimits^{\neq}} S(\alpha^* , 
	x_1, \ldots, x_n)  + c \hsp b ( \alpha(\cX^n), m^{[n]})
	\end{align*}
	for $m \geq n$ and $S_2(\alpha, \{ y_1, \ldots, y_m \}) = c \hsp b(\alpha
	(\cX^n), 0)$ for $m < n$ and with $c > 0$  is a consistent scoring
	function for the $n$-th factorial moment measure.
\end{enumerate}
Both $S_1$ and $S_2$ are strictly consistent if $S$ is strictly
consistent and $b$ is strict.
\end{prop}

In many cases of interest $\alpha^{(n)}$ is absolutely continuous with
respect to Lebesgue measure on $\cX^n$ and its density $\varrho^{(n)}$
is called \textit{product density}, see e.g.\ \citet{Chiuetal2013}.  A
(strictly) consistent scoring function for $\varrho^{(n)}$ can be
obtained from Proposition~\ref{prop:MomentM}~(ii) by choosing $S$ to
be a (strictly) consistent scoring function for densities.

\begin{example}  \label{ex:prod_log_and_quad}
Let  $n=2$ and for simplicity consider the product density
$\varrho^{(2)}$ of a stationary and isotropic point process.  In this
situation, $\varrho^{(2)}$  depends on the point distances only, i.e.\
it can be represented via $\varrho^{(2)} (x_1, x_2 ) = \varrho^{(2)}_0
( \Vert x_1 - x_2 \Vert)$ for some $\varrho^{(2)}_0 : [0, \infty) \to
[0, \infty)$.  Analogous to Example~4, we can use the quadratic
score for $b$ and the logarithmic score for $S$ in
Proposition~\ref{prop:MomentM}~(ii). This gives the strictly
consistent scoring function
\begin{align*}
S(\varrho^{(2)}, \{ y_1, \ldots, y_m\} ) &= - \underset{x_1, x_2 \in 
\{y_1, \ldots, y_m \}}{\sum\nolimits^{\neq}} \log (\varrho^{(2)}_0
(\Vert x_1 - x_2 \Vert) )  \\
&\phantom{=} + m^{[2]} \log \vert \varrho^{(2)} \vert + c \, 
( \vert \varrho^{(2)} \vert - m^{[2]})^2 ,
\end{align*}
where $c > 0$ is some scaling constant. Simulation experiments in
Section~\ref{subsec:sim_product} show how $S$ compares different
product density forecasts.
\end{example}

\subsection{Summary statistics}
\newsubsection{Summary statistics}

Summary statistics of point processes are central tools to quantify
point interactions such as clustering or inhibition.  This subsection
constructs strictly consistent scoring functions for the frequently
used $K$-function.  Throughout we assume that $\Phi$ is a
\textit{stationary} point process on $\real^d$, i.e.\ any translation
of the process by $x \in \real^d$, which we denote via $\Phi_x$, has
the same distribution as $\Phi$.  This implies that the intensity
measure of $\Phi$ is a multiple of Lebesgue measure and can be
represented via some $\lambda > 0$, see e.g.\
\citet[Chapter~4.1]{Chiuetal2013}.

A common way to describe a stationary point process is to consider its
properties in the neighbourhood of $x \in \real^d$, given that $x$ is a
point in $\Phi$.  Due to stationarity, the location of $x$ is
irrelevant and thus it is usually referred to as the ``typical point"
of $\Phi$.  The technical tool to describe the behaviour around this
point is the \textit{Palm distribution} of $\Phi$, denoted via $\P_0$
for probabilities and $\E_0$ for expectations.  It satisfies the
defining identity
\begin{align*}
\lambda \, \vert W \vert \, \E_0 f(\Phi) = \E \left( \sum_{x \in 
\Phi \cap W} f(\Phi_{-x} ) \right)
\end{align*}
for all measurable functions $f: \Mz \to \real$ such that the
expectations are finite, and it is independent of the observation
window $W \in \cB(\real^d)$ \citep[Chapter~4]{Illianetal2008}. When we
need to highlight the distribution of the point process, we write
$\E_{P,0}$ for the Palm expectation given $\Phi$ has distribution $P
\in \cP$. 

Denote the $d$-dimensional ball of radius $r> 0$ around zero via $B_r
= B(0,r)$. The \textit{K-function} of $\Phi$ is defined via
\begin{align*}
K : (0, \infty) \to [0, \infty) , \quad
r \mapsto \frac{\E_0  \Phi \left( B_r \backslash \{0\} \right)}
{\lambda} ,
\end{align*}
and it quantifies the mean number of points in a ball around the
``typical point" of $\Phi$, see e.g.\ \citet{Chiuetal2013}
and \citet{Illianetal2008} for details.  Deriving strictly
consistent scoring functions for the K-function appears challenging
since it combines the Palm distribution and the intensity. However, in
many situations both of these quantities are of interest.  We thus
derive a result which defines scoring functions for joint reports of
the $K$-function and the intensity. Our point process property of
interest is thus $\Gamma (P) := (\lambda_P , K_P)$, where the
subscript denotes the dependence of the quantities on the distribution
$P \in \cP$ of the process $\Phi$. Since observation windows are
always finite, we fix some $r^* > 0$ and let $K_P$ be the restriction
of the $K$-function to the interval $(0, r^*)$.

To derive consistent scoring functions let us fix some $r \in (0,
r^*)$ and assume for now that $\lambda_P$ is known and that instead
of data we directly observe the Palm distribution of~$\Phi$.  In this
simplified situation, $K_P(r)$ is just an expectation with respect to
$\P_0$, hence ``consistent scoring functions" for it are of the
Bregman form
\begin{align}   \label{eq:BregmanPalm}
S(x, \varphi) = - f(\lambda_P x) - f'(\lambda_P x) \big( \varphi (B_r \backslash \{0\}) -  \lambda_P x \big) ,
\end{align}
for a convex function $f : (0, \infty) \to \real$, see
Theorem~1 and Example~1.  This is because $\E_{P,0} b(x, \Phi) \ge
\E_{P,0} b(K_P(r), \Phi)$ holds for all $x \ge 0$ and $P \in \cP$.  To
arrive at a strictly consistent scoring function for the functional
$\Gamma$ three steps remain:  Firstly, we have to include a consistent
scoring function for the first component of $\Gamma$, i.e.\ the
intensity.  Moreover, we need to integrate~\eqref{eq:BregmanPalm} with
respect to $r$ in order to evaluate the $K$-function on the entire
interval $(0,r^*)$.  Finally, we have to account for the
fact that we can not observe $\P_0$, but only points of $\Phi$ on some
closed and bounded observation window $W \subset \real^d$.  Hence, we
need to compute the expected score $\E_0 S(x, \Phi)$ via an
expectation of $\Phi$ on $W$.  Such problems lead to edge corrections,
i.e.\ additional terms to account for the fact that (unobserved)
points outside of $W$ affect the estimation near the boundary of $W$,
see e.g.\ \citet[Chapter~4.7]{Chiuetal2013} for details. 
Since~\eqref{eq:BregmanPalm} is linear in $\varphi$, edge corrections
for the expected score are equivalent to edge corrections for the
expectation $\E_0 \Phi (B_r \backslash \{0\})$, which are well-known
in the context of $K$-function estimation.  Before we formalize these
three steps in a proposition, we state a result needed for the proof,
see \citet[Theorem~4]{Gneit2011}.

\begin{lemma}[revelation principle] \label{prop:RevelationPrinciple}
Let $\sA, \sA'$ be some sets and $g: \sA \rightarrow \sA'$ a bijection
with inverse $g^{-1}$.  Let $T: \cF \rightarrow \sA$ and $T_g : \cF
\rightarrow \sA'$ defined via $T_g(F) := g(T(F))$ be functionals.
Then $T$ is elicitable if and only if $T_g$ is elicitable.  A function
$S: \sA \times \sO \rightarrow \real$ is a (strictly) consistent
scoring function for $T$ if and only if $S_g : \sA' \times \sO
\rightarrow \real$, $(x,y) \mapsto S_g(x,y) := S( g^{-1}(x), y)$ is a
(strictly) consistent scoring function for $T_g$.
\end{lemma}

\begin{prop} \label{prop:Kfunction} Let
  $b_1, b_2 : [0, \infty) \times [0, \infty) \rightarrow \real$ be
  Bregman functions and $w: (0,\infty) \to [0,\infty)$ a weight
  function.  Define $\cC := \{ K_P \mid P \in \cP \}$, a set of
  possible $K$-functions, and let $\kappa$ satisfy
  $\E_P \kappa (B_r , \Phi \cap W) = \lambda_P \E_{P,0} \Phi (B_r
  \backslash \{0\} )$ for all $P \in \cP$ and $r \in (0, r^*)$.  Then
  the function $S: ((0,\infty) \times \cC) \times \Mz \to \real$
  defined via
\begin{align*}
S( (\lambda, K ), \varphi) = b_1 ( \lambda, \varphi (W) 
\vert W \vert^{-1} ) + \int_{0}^{r^*} b_2 (\lambda^2 K(r) , 
\kappa (B_r, \varphi) ) w(r) \dd r
\end{align*}
is consistent for the point process property $\Gamma (P) := 
(\lambda_P, K_P)$, where the second component is restricted to
$(0, r^*)$.  It is strictly consistent if $b_1$ and $b_2$ are strict
and $w$ is strictly positive.
\end{prop}

\begin{proof}
Using Theorem~1, the Fubini-Tonelli theorem, and 
\begin{align*}
\E_P \kappa (B_r , \Phi) = \lambda_P \E_{P,0} \Phi (B_r \backslash
\{0\} ) = \lambda_P^2 K_P (r)
\end{align*}
for $r \in (0,r^*)$, standard arguments show that the scoring function
\begin{align*}
S'( (\lambda, h) , \varphi) :=  b_1 ( \lambda, \varphi (W) 
\vert W \vert^{-1} ) + \int_{0}^{r^*} b_2 (h(r) , 
\kappa (B_r, \varphi) ) w(r)  \dd r ,
\end{align*}
where $h: (0, \infty) \to (0, \infty)$ is an increasing function, is
consistent for the property
$\Gamma' (P):= ( \lambda_P , \lambda_P^2 K_P(r) )$.  An application of
the revelation principle (Lemma~\ref{prop:RevelationPrinciple}) gives
(strict) consistency for $\Gamma$.
\end{proof}

Similar to Proposition~2, this result blends two scoring components,
namely the expected number of points and their distances.  Hence,
choosing suitable Bregman functions $b_1$ and $b_2$ in applications,
again leads to issues of balancing the magnitudes of different scoring
components.  A similarly intricate question is the choice of $\kappa$.
Relevant choices result from the construction of estimators for the
$K$-function, which are often based on dividing $\kappa$ by an
estimator for $\lambda^2$.  A common choice is
\begin{equation*}
\kappa_\mathrm{st} (B_r, \varphi) := \underset{x_1, x_2 \in \varphi
\cap W}{\sum\nolimits^{\neq}} \frac{ \mathbbm{1}_{B_r} (x_2 - x_1) }
{\vert W_{x_1} \cap W_{x_2} \vert},
\end{equation*}
where $W_z := \{ x + z \mid x \in W \}$ is the shifted observation
window and $r$ is such that $\vert W \cap W_z \vert $ is positive for
all $z \in B_r$, see e.g.\ \citet[Chapter~4.3]{Illianetal2008} and
\citet[Chapter~4.7]{Chiuetal2013}.  An alternative arises via
minus-sampling, i.e.\ by reducing the observation window $W$ in order
to reduce edge effects. This yields
\begin{equation*}
\kappa_\mathrm{minus} (B_r, \varphi) := \frac{1}{\vert W \vert}
\underset{x_1, x_2 \in \varphi \cap W, \, x_2 \in W \ominus r}
{\sum\nolimits^{\neq}}  \mathbbm{1}_{B_r} (x_2 - x_1) ,
\end{equation*}
where $W \ominus r := \{ x \mid B(x,r) \subset W \}$ is the reduced
observation window and $r < \diam (W) /2$. For other choices of
$\kappa$, most notably for isotropic point processes, see
\citet[Chapter~4.7]{Chiuetal2013}.

Practitioners usually rely on the $L$-function, a modification of the
$K$-function, which is defined via $L(r) = \sqrt[d]{ K(r) /\beta_d} $
for $r \geq 0$, where $\beta_d := \vert B_1 \vert$.  It satisfies
$L(r) = r$ for the Poisson point process, and thus normalizes the
$K$-function such that it is independent of the dimension~$d$ for a
Poisson point process \citep{Chiuetal2013}.  A (strictly) consistent
scoring function for the $L$-function follows immediately from
Proposition~\ref{prop:Kfunction} and another application of the
revelation principle.  The explicit formula follows by replacing the
first component of $b_2$ by $\lambda^2 L(r)^d \beta_d$  in
Proposition~\ref{prop:Kfunction}.  The idea underlying the
construction of scoring functions for the $K$- and $L$-function
presented here can be transferred to other summary statistics for
stationary point processes.

\section{Extended simulation study}
\newsection{Extended simulation study}

\subsection{Intensity}
\newsubsection{Intensity}
\label{subsec:sim_intensity}

This subsection extends Section~4.  We give more details on the used
point processes and provide a closer analysis of the simulation
experiments in the main paper.  We then perform additional simulations
with a different scoring function and study the approximation derived
in Proposition~4.

All experiments rely on the six intensities defined in Section~4, see
Figure~\ref{fig:plotIntensities} for an illustration.  We consider two
strictly consistent scoring functions for the intensity.  The first
choice is used in Section~4 and given by
\begin{align}   \label{eq:score_intensity1}
S_1(\Lambda, \{ y_1, \ldots, y_n\} ) = - \sum_{i=1}^{n} 
\log (\lambda (y_i) ) + n \log \vert \Lambda \vert + c \, 
( \vert \Lambda \vert - n)^2  ,
\end{align}
see also Example~3.5. Our second choice is
\begin{align}   \label{eq:score_intensity2} 
S_2( \Lambda, \{y_1, \ldots, y_n\} ) =  - \sum_{i=1}^{n} \log 
\lambda (y_i) + \int_\cX \lambda (y) \dd y  ,
\end{align}
which is defined in Proposition~5.1 and appears as the limit scoring
function in earthquake likelihood model testing, see Section~5.3.  The
scaling factor $c> 0$ in~\eqref{eq:score_intensity1} is set to
$c=1/10$.  We draw $N = \Nsim$  i.i.d.\ samples and repeat $M = \Msim$
times.

\begin{figure}[p]
\centering
\includegraphics[width=0.83\textwidth]{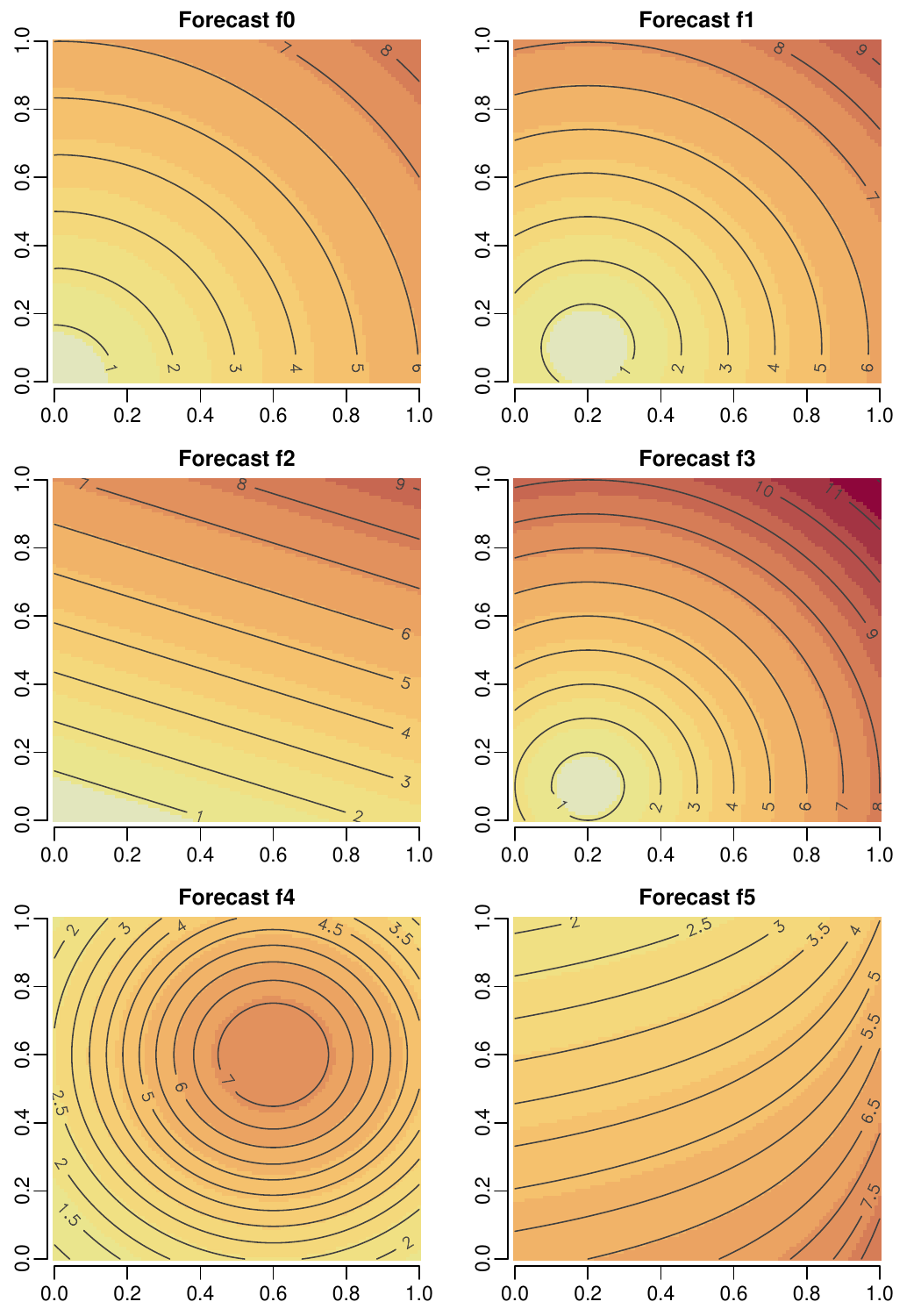}
\vspace{-.2cm}
\caption{Heat maps of the intensity forecasts $f_0, \ldots, f_5$, see Section~4.
\label{fig:plotIntensities}}
\end{figure}

\paragraph{Details on the point process models}

We consider four different data-generating processes for $\Phi$ on
$[0,1]^2$, all of which have (approximate) intensity $f_0 (x,y) = 6 
\sqrt{x^2 + y^2}$. The models are specified as follows:

\begin{enumerate}
\item An inhomogeneous Poisson point process with intensity $f_0$. 

\item A thinned Gaussian determinantal point process (DPP), see e.g.\
\citet{Houghetal2006} and \citet{Lavetal2015}.  In general, a DPP is a
locally finite point process with product densities (see 
Section~\ref{subsec:moment_measure}) given by
\begin{align*}
\varrho^{(n)} (x_1, \ldots, x_n) = \det \left( C(x_i, x_j) 
\right)_{i,j=1, \ldots, n}
\end{align*}
for $n \in \nat$, where $C : \real^d \times \real^d \to \real$ is a
covariance.  As a result, the DPP's intensity function is
$x \mapsto C(x,x)$ and it is stationary and isotropic whenever its
covariance is.  We choose $C(x_1,x_2) = C_0 (\Vert x_1 - x_2 \Vert )$,
where $C_0: [0,\infty) \to \real$ is the Gaussian covariance function
\begin{align}   \label{eq:cov_Gaussian}
C_0 (r) = \sigma^2 \exp \left\{ - \left( \frac{r}{s} \right)^2 \right\} ,
\end{align}
with variance $\sigma^2 = \max_{x,y \in [0,1]} f_0 (x,y)$ and scale
$s = 6/100$.  We then apply independent thinning to the homogeneous
Gaussian DPP in order to obtain the final point process with intensity
function $f_0$.

\item An inhomogeneous log-Gaussian Cox process (LGCP), see e.g.\
\citet[Chapter~6]{Illianetal2008}.  A LGCP is a Poisson point process
conditional on a random intensity function arising from a log-Gaussian
random field.  If $\mu: \real^d \to \real$ is the mean and $C: \real^d
\times \real^d \to \real$ is the covariance of the random field, then
the LGCP has intensity function
\begin{align*} 
x \mapsto \exp  \left( \mu (x) + \frac 12 C(x,x) 
\right) .
\end{align*}
We choose $C(x_1, x_2) = C_0 (\Vert x_1 - x_2 \Vert)$, where $C_0 :
[0,\infty) \to \real$ is the exponential covariance function
\begin{align}   \label{eq:cov_Exponential}
C_0 (r) = \sigma^2 \exp \left(- \frac{r}{s} \right) ,
\end{align}
with variance $\sigma^2 = 1/4$ and scale $s = 1/5$.  The mean is set
to $\mu(x) = \log ( f_0 (x) ) - 1/8$ such that the intensity equals
$f_0$.

\item  An inhomogeneous Thomas cluster process, see e.g.\
\citet[Chapter~6]{Illianetal2008}.  This is a cluster process which
arises from an inhomogeneous Poisson point process as parent and a
random number of cluster points which are drawn from a normal
distribution centred at its parent point.  As intensity of the parent
process we choose $2 f_0 / 3$ and the number of points per cluster
follows a Poisson distribution with parameter $3/2$.  The location of
each cluster point is determined by a normal distribution which is
centred at the parent point and where the components are uncorrelated
and have standard deviation $0.05$.  As a result of the clustering,
the intensity of the Thomas process is only approximately equal to
$f_0$.
\end{enumerate}

\begin{table}[hbt]
\centering
\caption{Fraction of replicates where the ``row forecast'' was preferred
over the ``column forecast'' by a standard DM test with level $\alpha =
\DMalpha$ based on the scoring function
$S_1$~\eqref{eq:score_intensity1}  and $M = 500$ replicates 
 \medskip
\label{tab:simPPPDPP}}
\begin{minipage}{.495\textwidth}
\centering
\textbf{Poisson}
\begin{tabular}{c cccccc }
 &  $f_0$ & $f_1$ & $f_2$ & $f_3$ & $f_4$ & $f_5$  \\
$f_0$ &   & \cellcolor{cyan!45}0.45 & \cellcolor{cyan!81}0.81 & \cellcolor{cyan!96}0.96 & \cellcolor{cyan!99}0.99 & \cellcolor{cyan!100}1.00 \\
$f_1$ & \cellcolor{cyan!0}0.00 &   & \cellcolor{cyan!40}0.40 & \cellcolor{cyan!88}0.88 & \cellcolor{cyan!82}0.82 & \cellcolor{cyan!99}0.99 \\
$f_2$ & \cellcolor{cyan!0}0.00 & \cellcolor{cyan!0}0.00 &   & \cellcolor{cyan!28}0.28 & \cellcolor{cyan!69}0.69 & \cellcolor{cyan!98}0.98 \\
$f_3$ & \cellcolor{cyan!0}0.00 & \cellcolor{cyan!0}0.00 & \cellcolor{cyan!1}0.01 &   & \cellcolor{cyan!24}0.24 & \cellcolor{cyan!91}0.91 \\
$f_4$ & \cellcolor{cyan!0}0.00 & \cellcolor{cyan!0}0.00 & \cellcolor{cyan!0}0.00 & \cellcolor{cyan!1}0.01 &   & \cellcolor{cyan!97}0.97 \\
$f_5$ & \cellcolor{cyan!0}0.00 & \cellcolor{cyan!0}0.00 & \cellcolor{cyan!0}0.00 & \cellcolor{cyan!0}0.00 & \cellcolor{cyan!0}0.00 &   \\
\end{tabular}

\end{minipage}
\begin{minipage}{.495\textwidth}
\centering
\textbf{DPP}
\begin{tabular}{c cccccc }
 &  $f_0$ & $f_1$ & $f_2$ & $f_3$ & $f_4$ & $f_5$  \\
$f_0$ &   & \cellcolor{cyan!52}0.52 & \cellcolor{cyan!83}0.83 & \cellcolor{cyan!97}0.97 & \cellcolor{cyan!99}0.99 & \cellcolor{cyan!100}1.00 \\
$f_1$ & \cellcolor{cyan!0}0.00 &   & \cellcolor{cyan!39}0.39 & \cellcolor{cyan!91}0.91 & \cellcolor{cyan!80}0.80 & \cellcolor{cyan!100}1.00 \\
$f_2$ & \cellcolor{cyan!0}0.00 & \cellcolor{cyan!0}0.00 &   & \cellcolor{cyan!27}0.27 & \cellcolor{cyan!67}0.67 & \cellcolor{cyan!97}0.97 \\
$f_3$ & \cellcolor{cyan!0}0.00 & \cellcolor{cyan!0}0.00 & \cellcolor{cyan!1}0.01 &   & \cellcolor{cyan!22}0.22 & \cellcolor{cyan!93}0.93 \\
$f_4$ & \cellcolor{cyan!0}0.00 & \cellcolor{cyan!0}0.00 & \cellcolor{cyan!0}0.00 & \cellcolor{cyan!1}0.01 &   & \cellcolor{cyan!98}0.98 \\
$f_5$ & \cellcolor{cyan!0}0.00 & \cellcolor{cyan!0}0.00 & \cellcolor{cyan!0}0.00 & \cellcolor{cyan!0}0.00 & \cellcolor{cyan!0}0.00 &   \\
\end{tabular}

\end{minipage}

\smallskip
\begin{minipage}{.495\textwidth}
\centering
\textbf{LGCP}
\begin{tabular}{c cccccc }
 &  $f_0$ & $f_1$ & $f_2$ & $f_3$ & $f_4$ & $f_5$  \\
$f_0$ &   & \cellcolor{cyan!48}0.48 & \cellcolor{cyan!80}0.80 & \cellcolor{cyan!93}0.93 & \cellcolor{cyan!99}0.99 & \cellcolor{cyan!100}1.00 \\
$f_1$ & \cellcolor{cyan!0}0.00 &   & \cellcolor{cyan!39}0.39 & \cellcolor{cyan!85}0.85 & \cellcolor{cyan!81}0.81 & \cellcolor{cyan!100}1.00 \\
$f_2$ & \cellcolor{cyan!0}0.00 & \cellcolor{cyan!0}0.00 &   & \cellcolor{cyan!27}0.27 & \cellcolor{cyan!66}0.66 & \cellcolor{cyan!97}0.97 \\
$f_3$ & \cellcolor{cyan!0}0.00 & \cellcolor{cyan!0}0.00 & \cellcolor{cyan!1}0.01 &   & \cellcolor{cyan!19}0.19 & \cellcolor{cyan!92}0.92 \\
$f_4$ & \cellcolor{cyan!0}0.00 & \cellcolor{cyan!0}0.00 & \cellcolor{cyan!0}0.00 & \cellcolor{cyan!1}0.01 &   & \cellcolor{cyan!97}0.97 \\
$f_5$ & \cellcolor{cyan!0}0.00 & \cellcolor{cyan!0}0.00 & \cellcolor{cyan!0}0.00 & \cellcolor{cyan!0}0.00 & \cellcolor{cyan!0}0.00 &   \\
\end{tabular}

\end{minipage}
\begin{minipage}{.495\textwidth}
\centering
\textbf{Thomas}
\begin{tabular}{c cccccc }
 &  $f_0$ & $f_1$ & $f_2$ & $f_3$ & $f_4$ & $f_5$  \\
$f_0$ &   & \cellcolor{cyan!24}0.24 & \cellcolor{cyan!52}0.52 & \cellcolor{cyan!76}0.76 & \cellcolor{cyan!89}0.89 & \cellcolor{cyan!100}1.00 \\
$f_1$ & \cellcolor{cyan!0}0.00 &   & \cellcolor{cyan!26}0.26 & \cellcolor{cyan!60}0.60 & \cellcolor{cyan!56}0.56 & \cellcolor{cyan!91}0.91 \\
$f_2$ & \cellcolor{cyan!0}0.00 & \cellcolor{cyan!0}0.00 &   & \cellcolor{cyan!18}0.18 & \cellcolor{cyan!44}0.44 & \cellcolor{cyan!81}0.81 \\
$f_3$ & \cellcolor{cyan!0}0.00 & \cellcolor{cyan!0}0.00 & \cellcolor{cyan!1}0.01 &   & \cellcolor{cyan!14}0.14 & \cellcolor{cyan!68}0.68 \\
$f_4$ & \cellcolor{cyan!0}0.00 & \cellcolor{cyan!0}0.00 & \cellcolor{cyan!0}0.00 & \cellcolor{cyan!1}0.01 &   & \cellcolor{cyan!78}0.78 \\
$f_5$ & \cellcolor{cyan!0}0.00 & \cellcolor{cyan!0}0.00 & \cellcolor{cyan!0}0.00 & \cellcolor{cyan!0}0.00 & \cellcolor{cyan!0}0.00 &   \\
\end{tabular}

\end{minipage}
\end{table}

\paragraph{Further details for the experiments of Section~4}

Section~4 presents four simulation experiments based on the scoring
function $S_1$.  Table~\ref{tab:simPPPDPP} shows the results of DM
tests (see \citet{DiebMari1995} and Section~\ref{sec:scenarios}) for
these experiments.  For each of the $M = \Msim$ realizations we test
whether forecast $f_i$ (row) achieves the same expected score as
forecast $f_j$ (column).  The rejection frequencies in favour of $f_0$
against $f_j$, $j = 1, \ldots, 5$ (first row of each table) are
generally in line with the mean score differences in Figure~1.
Moreover, the results of the DM tests are similar for all four
simulation experiments.  In the third and fourth experiment (lower
part of Table~\ref{tab:simPPPDPP}) the frequencies of rejection in
favour of the optimal forecast $f_0$ (first row of each table)
decrease slightly for the LGCP and substantially for the Thomas
process.  An intuitive reason for this is that clustering, which is a
feature of both processes, complicates the distinction between
different intensity forecasts.

\paragraph{Experiments with a different scoring function}

We now investigate how the forecast comparison changes when using the
scoring function $S_2$ instead of the scoring function $S_1$ from
Section~4.  Boxplots of mean score differences are given in
Figure~\ref{fig:simALL4_pois} and they are generally similar to the
ones presented in Figure~1.

\begin{figure}[bht]
\centering
\includegraphics[width = \textwidth]{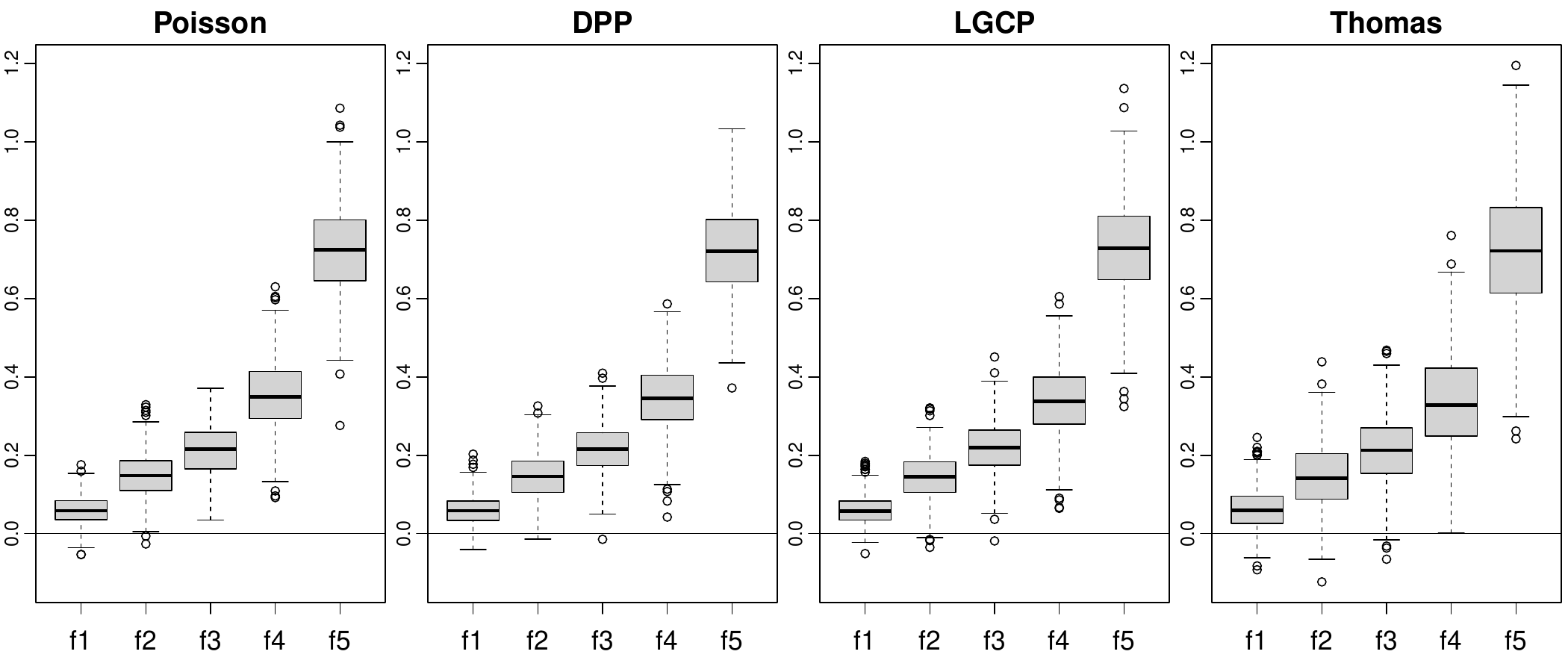}
\vspace{-.8cm}
\caption{Boxplot of difference in mean scores $\bar s_j - \bar s_0$
for $j=1, \ldots, 5$ and scoring function
$S_2$~\eqref{eq:score_intensity2}.  From left to right, $\Phi$ is a
Poisson point process, a Gaussian determinantal point process, a
log-Gaussian Cox process, or an inhomogeneous Thomas process.
Means are based on $N = \Nsim$ realizations, boxplots on $M = \Msim$
replicates.
\label{fig:simALL4_pois}}
\end{figure}

The same conclusion holds for the results of DM tests given in
Table~\ref{tab:simPoissonCox_pois} resemble those in
Table~\ref{tab:simPPPDPP}.  This suggests that in our experiments the
choice of $c = 1/10$ for $S_1$ leads to a similar balance of shape and
total mass of the intensity as with $S_2$.  However, in other forecast
settings, or with a different choice of $c$, the two scoring functions
may lead to differing conclusions.  As in the previous experiments,
the clustering of the LGCP and the Thomas process leads to less
conclusive decisions between the forecasts.  In contrast, the
inhibition of the Gaussian DPP seems to facilitate the comparison
between the forecasts.

A further sequence of experiments considers the speed of convergence
in Proposition~4, i.e.\ how well score differences based on
$S_\mathrm{cell}^{\cT_n}$, as defined in~(14), approximate score
differences based on $S_2$~\eqref{eq:score_intensity2}.  We select a
family of partitions $(\cT_n)_{n \in \nat}$ of $[0,1]^2$ which arises
from dyadic partitions of both axes.  Specifically, each grid cell
$B_{ij}^{(n)} \in \cT_n$ is given by
$[(i-1)/2^n, i/2^n ] \times [(j-1)/2^n , j/2^n]$ for
$i,j \in \{1, \ldots, 2^n \}$.  The number of cells is thus
$k_n = 2^{2n}$ and we choose $n \in \{1, \ldots, 6\}$ for the
simulations.  As forecasts we rely on the intensity functions
$f_0, \ldots, f_5$ introduced in Section~4 which we transform into
grid-based reports $f_{l,ij}^{(n)}$ by integrating $f_l$ over the grid
cell $B_{ij}^{(n)}$.  These reports are then compared to the number of
points per cells via $S_\mathrm{cell}^{\cT_n}$.  We study the
convergence of the rejection probabilities of DM tests based on
$S_\mathrm{cell}^{\cT_n}$ for $N = \Nsim$ i.i.d.\ samples of $\Phi$
and increasing $n$.  The corresponding fractions converge to the
values in Table~\ref{tab:simPoissonCox_pois}, as illustrated in
Figure~\ref{fig:simDMconvergence} for the comparisons of $f_0$ to
$f_1, \ldots, f_5$.  These simulations suggest that for forecasts
which are far from the underlying truth $n=2$, i.e.\ 16 grid cells, is
already enough to obtain DM results based on $S_\mathrm{cell}^{\cT_n}$
which are in good agreement with the results based on $S_2$
(Table~\ref{tab:simPoissonCox_pois}).  For intensity functions closer
to the truth, such as $f_1$, $n = 3$, i.e.~64 grid cells, seems
necessary to obtain a good approximation.

\begin{table}[h]
\centering
\caption{Fraction of replicates where the ``row forecast'' was preferred
over the ``column forecast'' by a standard DM test with level $\alpha =
\DMalpha$ based on the scoring function
$S_2$~\eqref{eq:score_intensity2} and $M = 500$ replicates 
 \medskip
\label{tab:simPoissonCox_pois}}
\begin{minipage}{.495\textwidth}
\centering
\textbf{Poisson} \\
\begin{tabular}{c cccccc }
 &  $f_0$ & $f_1$ & $f_2$ & $f_3$ & $f_4$ & $f_5$  \\
$f_0$ &   & \cellcolor{cyan!46}0.46 & \cellcolor{cyan!84}0.84 & \cellcolor{cyan!95}0.95 & \cellcolor{cyan!99}0.99 & \cellcolor{cyan!100}1.00 \\
$f_1$ & \cellcolor{cyan!0}0.00 &   & \cellcolor{cyan!48}0.48 & \cellcolor{cyan!87}0.87 & \cellcolor{cyan!84}0.84 & \cellcolor{cyan!100}1.00 \\
$f_2$ & \cellcolor{cyan!0}0.00 & \cellcolor{cyan!0}0.00 &   & \cellcolor{cyan!24}0.24 & \cellcolor{cyan!70}0.70 & \cellcolor{cyan!98}0.98 \\
$f_3$ & \cellcolor{cyan!0}0.00 & \cellcolor{cyan!0}0.00 & \cellcolor{cyan!1}0.01 &   & \cellcolor{cyan!29}0.29 & \cellcolor{cyan!94}0.94 \\
$f_4$ & \cellcolor{cyan!0}0.00 & \cellcolor{cyan!0}0.00 & \cellcolor{cyan!0}0.00 & \cellcolor{cyan!0}0.00 &   & \cellcolor{cyan!96}0.96 \\
$f_5$ & \cellcolor{cyan!0}0.00 & \cellcolor{cyan!0}0.00 & \cellcolor{cyan!0}0.00 & \cellcolor{cyan!0}0.00 & \cellcolor{cyan!0}0.00 &   \\
\end{tabular}

\end{minipage}
\begin{minipage}{.495\textwidth}
\centering
\textbf{DPP}
\begin{tabular}{c cccccc }
 &  $f_0$ & $f_1$ & $f_2$ & $f_3$ & $f_4$ & $f_5$  \\
$f_0$ &   & \cellcolor{cyan!47}0.47 & \cellcolor{cyan!84}0.84 & \cellcolor{cyan!96}0.96 & \cellcolor{cyan!99}0.99 & \cellcolor{cyan!100}1.00 \\
$f_1$ & \cellcolor{cyan!0}0.00 &   & \cellcolor{cyan!43}0.43 & \cellcolor{cyan!93}0.93 & \cellcolor{cyan!86}0.86 & \cellcolor{cyan!100}1.00 \\
$f_2$ & \cellcolor{cyan!0}0.00 & \cellcolor{cyan!0}0.00 &   & \cellcolor{cyan!22}0.22 & \cellcolor{cyan!69}0.69 & \cellcolor{cyan!98}0.98 \\
$f_3$ & \cellcolor{cyan!0}0.00 & \cellcolor{cyan!0}0.00 & \cellcolor{cyan!1}0.01 &   & \cellcolor{cyan!28}0.28 & \cellcolor{cyan!93}0.93 \\
$f_4$ & \cellcolor{cyan!0}0.00 & \cellcolor{cyan!0}0.00 & \cellcolor{cyan!0}0.00 & \cellcolor{cyan!1}0.01 &   & \cellcolor{cyan!97}0.97 \\
$f_5$ & \cellcolor{cyan!0}0.00 & \cellcolor{cyan!0}0.00 & \cellcolor{cyan!0}0.00 & \cellcolor{cyan!0}0.00 & \cellcolor{cyan!0}0.00 &   \\
\end{tabular}

\end{minipage}

\smallskip
\begin{minipage}{.495\textwidth}
\centering
\textbf{LGCP}
\begin{tabular}{c cccccc }
 &  $f_0$ & $f_1$ & $f_2$ & $f_3$ & $f_4$ & $f_5$  \\
$f_0$ &   & \cellcolor{cyan!44}0.44 & \cellcolor{cyan!79}0.79 & \cellcolor{cyan!93}0.93 & \cellcolor{cyan!98}0.98 & \cellcolor{cyan!100}1.00 \\
$f_1$ & \cellcolor{cyan!0}0.00 &   & \cellcolor{cyan!39}0.39 & \cellcolor{cyan!84}0.84 & \cellcolor{cyan!83}0.83 & \cellcolor{cyan!99}0.99 \\
$f_2$ & \cellcolor{cyan!0}0.00 & \cellcolor{cyan!0}0.00 &   & \cellcolor{cyan!27}0.27 & \cellcolor{cyan!68}0.68 & \cellcolor{cyan!96}0.96 \\
$f_3$ & \cellcolor{cyan!0}0.00 & \cellcolor{cyan!0}0.00 & \cellcolor{cyan!0}0.00 &   & \cellcolor{cyan!23}0.23 & \cellcolor{cyan!92}0.92 \\
$f_4$ & \cellcolor{cyan!0}0.00 & \cellcolor{cyan!0}0.00 & \cellcolor{cyan!0}0.00 & \cellcolor{cyan!1}0.01 &   & \cellcolor{cyan!97}0.97 \\
$f_5$ & \cellcolor{cyan!0}0.00 & \cellcolor{cyan!0}0.00 & \cellcolor{cyan!0}0.00 & \cellcolor{cyan!0}0.00 & \cellcolor{cyan!0}0.00 &   \\
\end{tabular}

\end{minipage}
\begin{minipage}{.495\textwidth}
\centering
\textbf{Thomas}
\begin{tabular}{c cccccc }
 &  $f_0$ & $f_1$ & $f_2$ & $f_3$ & $f_4$ & $f_5$  \\
$f_0$ &   & \cellcolor{cyan!24}0.24 & \cellcolor{cyan!53}0.53 & \cellcolor{cyan!73}0.73 & \cellcolor{cyan!86}0.86 & \cellcolor{cyan!99}0.99 \\
$f_1$ & \cellcolor{cyan!0}0.00 &   & \cellcolor{cyan!26}0.26 & \cellcolor{cyan!58}0.58 & \cellcolor{cyan!53}0.53 & \cellcolor{cyan!92}0.92 \\
$f_2$ & \cellcolor{cyan!0}0.00 & \cellcolor{cyan!1}0.01 &   & \cellcolor{cyan!16}0.16 & \cellcolor{cyan!42}0.42 & \cellcolor{cyan!80}0.80 \\
$f_3$ & \cellcolor{cyan!0}0.00 & \cellcolor{cyan!0}0.00 & \cellcolor{cyan!1}0.01 &   & \cellcolor{cyan!15}0.15 & \cellcolor{cyan!69}0.69 \\
$f_4$ & \cellcolor{cyan!0}0.00 & \cellcolor{cyan!0}0.00 & \cellcolor{cyan!0}0.00 & \cellcolor{cyan!1}0.01 &   & \cellcolor{cyan!79}0.79 \\
$f_5$ & \cellcolor{cyan!0}0.00 & \cellcolor{cyan!0}0.00 & \cellcolor{cyan!0}0.00 & \cellcolor{cyan!0}0.00 & \cellcolor{cyan!0}0.00 &   \\
\end{tabular}

\end{minipage}
\end{table}

\begin{figure}[p]
\centering
\includegraphics[width = \textwidth]{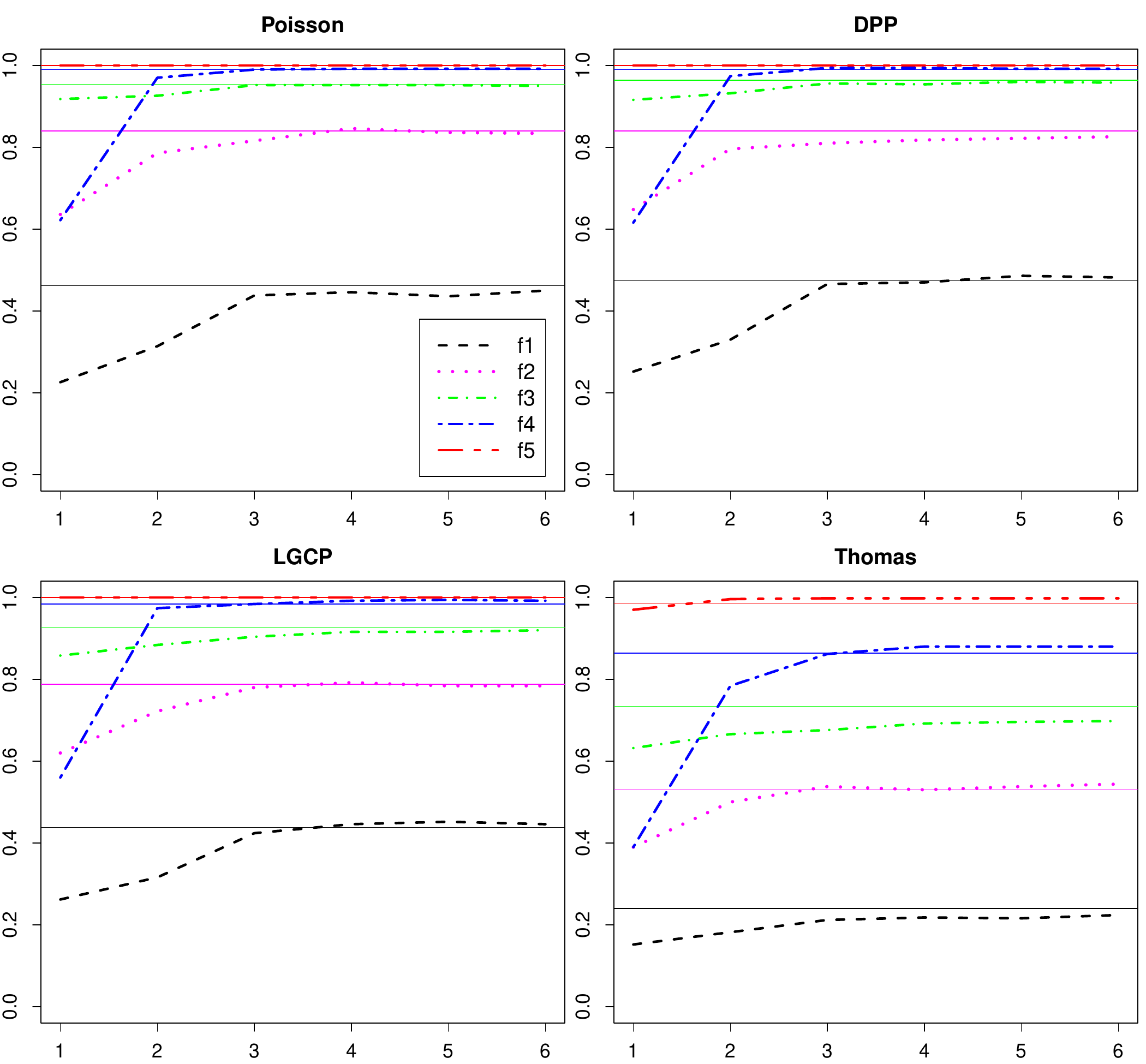}
\caption{Fraction of replicates where $f_0$ was preferred over $f_1, \ldots, f_5$
by a standard DM test with level $\alpha = \DMalpha$ based on the
scoring function $S_\mathrm{cell}^{\cT_n}$, with $n$ varying along the horizontal axis,
sample size $N = \Nsim$, and $M = \Msim$ replicates.  The solid lines represent
the fractions resulting from the use of $S_2$ 
(see~\eqref{eq:score_intensity2}), as given in 
Table~\ref{tab:simPoissonCox_pois}.  The legend in the upper left plot
applies to all other plots, too.
\label{fig:simDMconvergence}}
\end{figure}

\subsection{Product density}
\newsubsection{Product density}
\label{subsec:sim_product}

This subsection presents simulation experiments for the product
density (Section~\ref{subsec:moment_measure}).  We simulate stationary and
isotropic point processes with three different second order structures
corresponding to inhibition, clustering, and no interaction.  We draw
$N=30$  i.i.d.\ samples $\varphi_i$ from $\Phi$ and compare the mean
scores for different forecasts, in the same way as in Section~4.  The 
scoring function $S$ is defined in Example~\ref{ex:prod_log_and_quad}
and the scaling factor $c = 10^{-5}$ is chosen such that the log and
squared terms are of the same order of magnitude.  We repeat the
simulations $M=500$ times to assess the variation in mean scores.

\paragraph{Details on the point process models}

We simulate three different stationary and isotropic data-generating
processes $\Phi$ on the window $[0,1]^2$ with intensity $\lambda =
25$.  The models are specified as follows:

\begin{enumerate}

\item A LGCP which is determined by a stationary and isotropic
Gaussian process with mean $\mu \in \real$ and covariance function
$C_0$, see e.g.\ \citet{Illianetal2008}.  Its second order product
density $\varrho^{(2)} : \real^d \times \real^d \to \real$ is given by
$\varrho^{(2)} (x_1, x_2) = \varrho^{(2)}_0 (\Vert x_1 - x_2 \Vert )$,
where
\begin{align*}
\varrho^{(2)}_0 (r) = \exp \left( 2 \mu + C_0(0) + C_0(r) \right) .
\end{align*}
We choose $C_0$ as the Gaussian covariance
function~\eqref{eq:cov_Gaussian} with variance $\sigma^2 = \log 2$ and
scale $s = 5/100$ and set $\mu = \log (\lambda) - \sigma^2 /2$.

\item A homogeneous Poisson point process.

\item A DPP defined via the Gaussian covariance 
function~\eqref{eq:cov_Gaussian}, see e.g.\ \citet{Houghetal2006} and
\citet{Lavetal2015}.  Its second order product density is given by 
$\varrho^{(2)} (x_1, x_2) = \varrho^{(2)}_0 (\Vert x_1 - x_2 \Vert )$,
where
\begin{align*}
\varrho^{(2)}_0 (r) = C_0 (0)^2 - C_0 (r)^2 ,
\end{align*}
and $C_0$ is the Gaussian covariance~\eqref{eq:cov_Gaussian} with
variance $\sigma^2 = \lambda^2$ and scale $s=0.06$.
\end{enumerate}

\paragraph{Forecast comparison}

The three simulation experiments compare five different
product density forecasts, which are based on stationary and isotropic
point processes, see Example~\ref{ex:prod_log_and_quad}.  Hence, the
forecasts take the form $\varrho^{(2)} (x_1, x_2 ) = \varrho^{(2)}_0
(\Vert x_1 - x_2 \Vert)$, with the function $\varrho^{(2)}_0$ given by
\begin{align*}
f_1 (r) &= \exp \left[ 2 \mu + \sigma^2 \left\{ 1 + \exp( - 400 r^2 )
\right\} \right] \\
f_2 (r) &= \exp \left[ 2 \mu + \sigma^2 \left\{ 1 + \exp( - 20 r ) 
\right\} \right] \\
f_3 (r) &= \lambda^2 \\
f_4 (r) &= \lambda^2 \left\{ 1 - \exp ( - 2 r/s ) \right\} \\
f_5 (r) &= \lambda^2 \left\{ 1 - \exp ( - 2 (r /s)^2 ) \right\} ,
\end{align*}
where $\mu = \log (\lambda) - \sigma^2 /2$, $\sigma^2 = \log(2)$,
$s = 0.06$, and $\lambda = 25$.  See Figure~\ref{fig:plotProdDens} for
a graphical comparison of the different functions.  The forecasts
$f_1$ and $f_2$ represent clustering, since they arise as product
densities of LGCPs with Gaussian or exponential covariance function
(see~\eqref{eq:cov_Gaussian} and~\eqref{eq:cov_Exponential}).  The
constant function $f_3$ corresponds to a homogeneous Poisson process.
The forecasts $f_4$ and $f_5$ arise as product densities of DPPs with
Gaussian or exponential covariance function and thus represent
inhibition.  Our parameter choices ensure that the point process
models corresponding to $f_1, \ldots, f_5$ all have intensity equal
to $\lambda$, so forecast misspecifications only occur in the product
density.

\begin{figure}[htb]
\centering
\includegraphics[width = 0.8\textwidth]{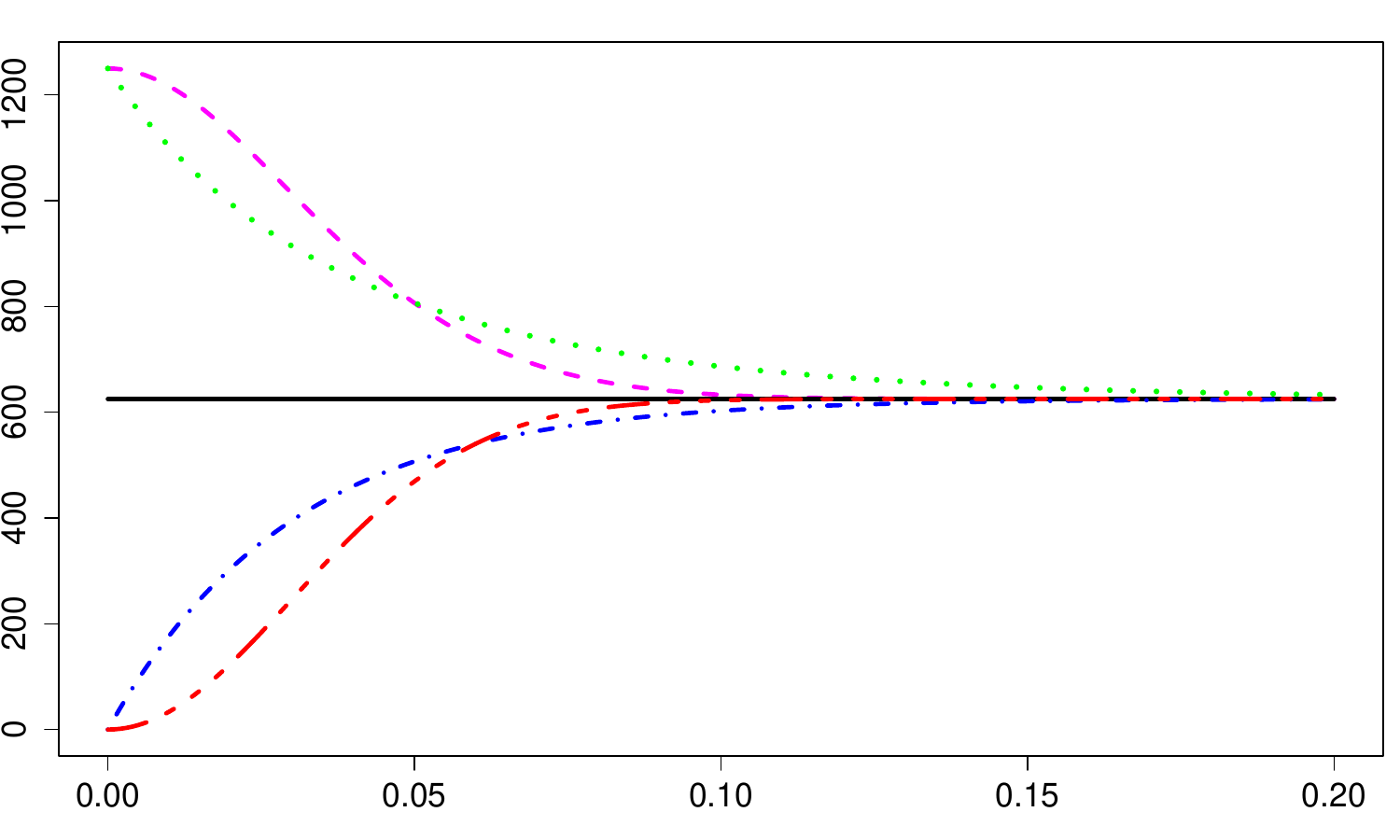}
\caption{Plot of the five different choices for $\varrho^{(2)}_0 : 
[0,\infty) \to [0, \infty)$ on which the product density forecasts in
Section~\ref{subsec:sim_product} are based.  The first two ($f_1$ and
$f_2$) represent clustering, the last two ($f_4$ and $f_5$)
inhibition.  The constant $f_3$ implies no interaction.
\label{fig:plotProdDens}}
\end{figure}

\begin{figure}[htb]
\centering
\includegraphics[width = \textwidth]{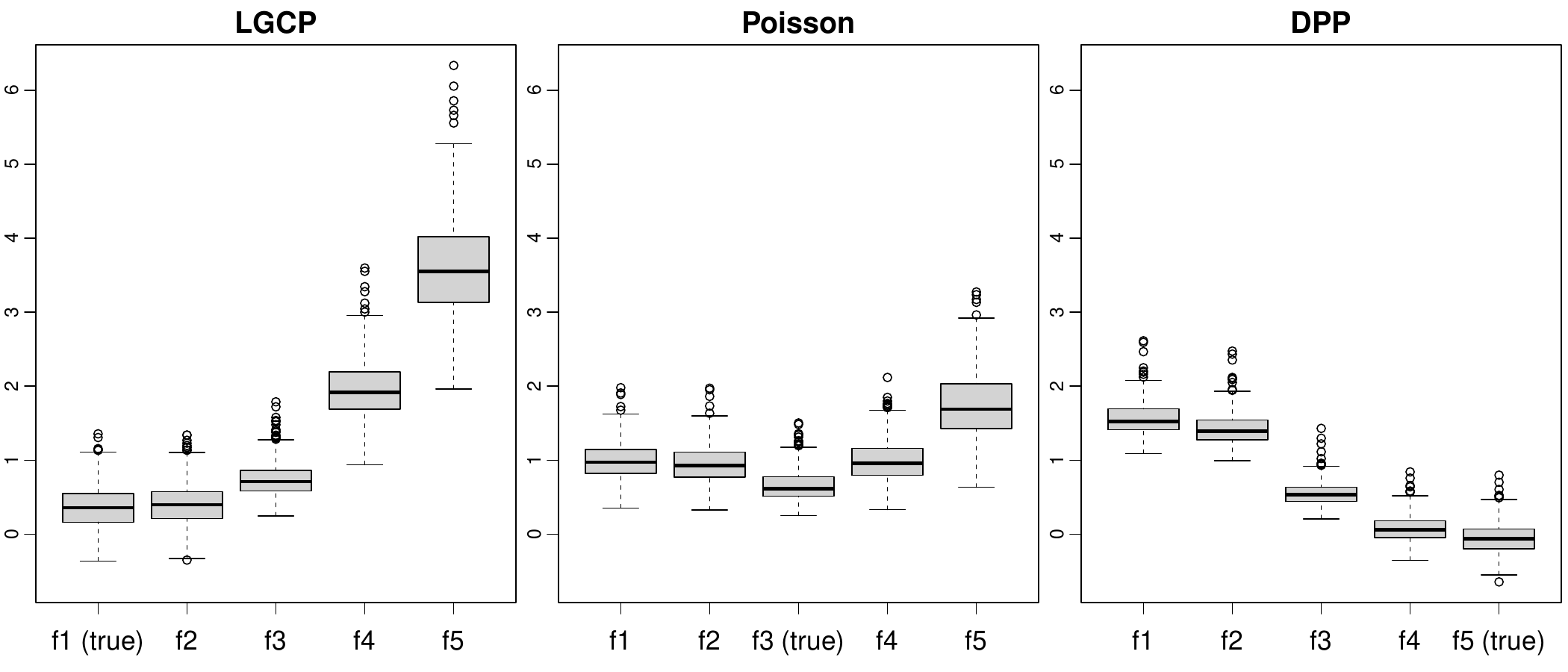}
\vspace{-.8cm}
\caption{Boxplots of mean scores $\bar s_j$ for different product
density forecasts, where $\Phi$ is a log-Gaussian Cox process (left),
a homogeneous Poisson process (centre), or a Gaussian determinantal
point process (right).  Means are based on $N = 30$ realizations,
boxplots on $M = 500$ replicates. 
\label{fig:simPPPCox}}
\end{figure}

In the first experiment the true $\Phi$ is a LGCP with a Gaussian
covariance function such that its product density corresponds to
$f_1$.  In the second experiment $\Phi$ is a homogeneous Poisson
process with intensity $\lambda$, such that $f_3$ becomes the optimal
forecast in this situation.  Lastly, we let $\Phi$ be a DPP with
Gaussian covariance function and parameters such that $f_5$ is
optimal.  We thus perform one experiment for each of the three
phenomena clustering, no interaction, and inhibition.

The simulated mean scores are displayed in Figure~\ref{fig:simPPPCox}
for all three experiments.  The optimal forecast consistently achieves
the lowest mean score.  In the case of clustering (left subfigure) the
LGCP related forecasts $f_1$ and $f_2$ perform roughly similar, while
the misspecified no interaction and inhibition forecasts $f_3$, $f_4$
and $f_5$ lead to considerably higher mean scores.  A similar, but
mirrored behaviour is apparent in the inhibition experiment (right
subfigure): The forecast $f_4$, which gets the nature of point
interactions right, attains low mean scores, even though it is not
optimal.  The mean scores of the Poisson forecast $f_3$ are always in
between the ``extremes''.  The DM test probabilities of the three
experiments are given in Table~\ref{tab:simPoissonCox_product} and
support these observations.  Additionally, the DM results illustrate
that the clustering forecasts $f_1$ and $f_2$ are preferred more often
over the inhibition forecast $f_5$ in the case of Poisson data (centre
table).

\begin{table}[htb]
\centering
\caption{Fraction of times the ``row forecast'' was preferred over the
``column forecast'' by a standard DM test with level $\alpha = 
\DMalpha$ in the product density experiments
(Section~\ref{subsec:sim_product}), based on $M = \Msim$ repetitions
\medskip
\label{tab:simPoissonCox_product}}
\begin{minipage}{.495\textwidth}
\centering
\textbf{LGCP} \\
\begin{tabular}{c ccccc }
 &  $f_1$ & $f_2$ & $f_3$ & $f_4$ & $f_5$  \\
$f_1$ &   & \cellcolor{cyan!18}0.18 & \cellcolor{cyan!63}0.63 & \cellcolor{cyan!99}0.99 & \cellcolor{cyan!100}1.00 \\
$f_2$ & \cellcolor{cyan!0}0.00 &   & \cellcolor{cyan!57}0.57 & \cellcolor{cyan!99}0.99 & \cellcolor{cyan!100}1.00 \\
$f_3$ & \cellcolor{cyan!0}0.00 & \cellcolor{cyan!0}0.00 &   & \cellcolor{cyan!100}1.00 & \cellcolor{cyan!100}1.00 \\
$f_4$ & \cellcolor{cyan!0}0.00 & \cellcolor{cyan!0}0.00 & \cellcolor{cyan!0}0.00 &   & \cellcolor{cyan!100}1.00 \\
$f_5$ & \cellcolor{cyan!0}0.00 & \cellcolor{cyan!0}0.00 & \cellcolor{cyan!0}0.00 & \cellcolor{cyan!0}0.00 &   \\
\end{tabular}

\end{minipage}
\begin{minipage}{.495\textwidth}
\centering
\textbf{DPP} \\
\begin{tabular}{c ccccc }
 &  $f_1$ & $f_2$ & $f_3$ & $f_4$ & $f_5$  \\
$f_1$ &   & \cellcolor{cyan!0}0.00 & \cellcolor{cyan!0}0.00 & \cellcolor{cyan!0}0.00 & \cellcolor{cyan!0}0.00 \\
$f_2$ & \cellcolor{cyan!90}0.90 &   & \cellcolor{cyan!0}0.00 & \cellcolor{cyan!0}0.00 & \cellcolor{cyan!0}0.00 \\
$f_3$ & \cellcolor{cyan!100}1.00 & \cellcolor{cyan!100}1.00 &   & \cellcolor{cyan!0}0.00 & \cellcolor{cyan!0}0.00 \\
$f_4$ & \cellcolor{cyan!100}1.00 & \cellcolor{cyan!100}1.00 & \cellcolor{cyan!100}1.00 &   & \cellcolor{cyan!0}0.00 \\
$f_5$ & \cellcolor{cyan!100}1.00 & \cellcolor{cyan!100}1.00 & \cellcolor{cyan!96}0.96 & \cellcolor{cyan!57}0.57 &   \\
\end{tabular}

\end{minipage}

\smallskip
\textbf{Poisson} \\
\begin{tabular}{c ccccc }
 &  $f_1$ & $f_2$ & $f_3$ & $f_4$ & $f_5$  \\
$f_1$ &   & \cellcolor{cyan!1}0.01 & \cellcolor{cyan!0}0.00 & \cellcolor{cyan!4}0.04 & \cellcolor{cyan!43}0.43 \\
$f_2$ & \cellcolor{cyan!17}0.17 &   & \cellcolor{cyan!0}0.00 & \cellcolor{cyan!6}0.06 & \cellcolor{cyan!50}0.50 \\
$f_3$ & \cellcolor{cyan!77}0.77 & \cellcolor{cyan!68}0.68 &   & \cellcolor{cyan!64}0.64 & \cellcolor{cyan!96}0.96 \\
$f_4$ & \cellcolor{cyan!7}0.07 & \cellcolor{cyan!5}0.05 & \cellcolor{cyan!0}0.00 &   & \cellcolor{cyan!100}1.00 \\
$f_5$ & \cellcolor{cyan!0}0.00 & \cellcolor{cyan!0}0.00 & \cellcolor{cyan!0}0.00 & \cellcolor{cyan!0}0.00 &   \\
\end{tabular}

\end{table}

\section{Additional details for the case study}
\newsection{Additional details for the case study}

This material extends Section~5.2.  Figure~\ref{fig:plot_quad_time}
reproduces Figure~3 but with the quadratic score $S_\mathrm{quad}$
rather than the Poisson score $S_\mathrm{pois}$.  In contrast to
Figure~3 we see that there are periods without events where the LG
model rather than the FMC model attains the lowest scores.

Figures~\ref{fig:spatial_comp_LG} and~\ref{fig:spatial_comp_SMA} use
the same methods as in Figure~4 to compare the LM model to the LG and
the SMA model.  The regions of superior or inferior forecast
performance of the LM model remain generally the same across the three
comparisons.  The right plots of these figures compare the forecasts
after spatial aggregation, for which we give details now.

\begin{figure}[htb]
\centering
\includegraphics[width = 0.7\textwidth]{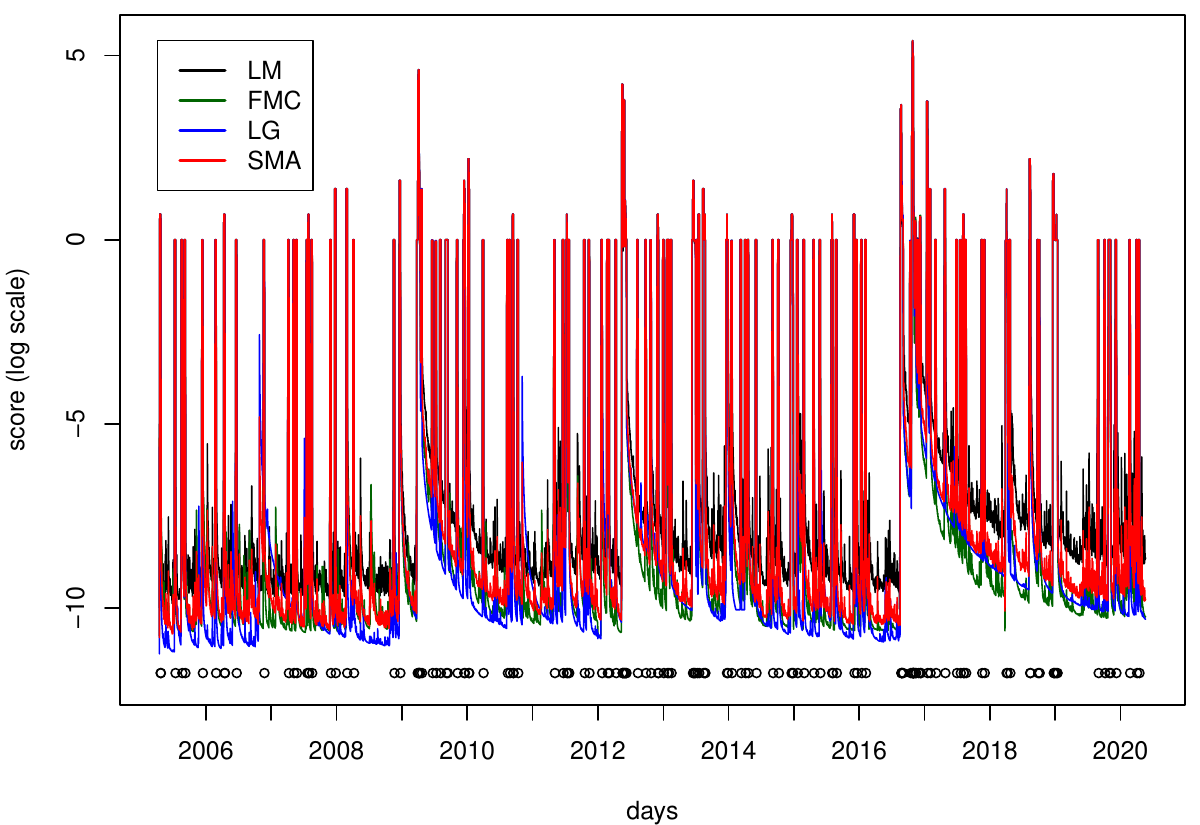}
\caption{Daily scores $s_{j,t}$ from (10) based on $S_\mathrm{quad}$ for the four forecasting models from 2005 to 2020, logarithmic scale.  The circles indicate the days of M4+ earthquakes and the tickmarks on the horizontal axis mark the first day of each year.
\label{fig:plot_quad_time}}
\end{figure}

\begin{figure}[p]
\centering
\begin{minipage}{.495\textwidth}
	\includegraphics[clip, width = \textwidth, trim = 8 0 10 0]{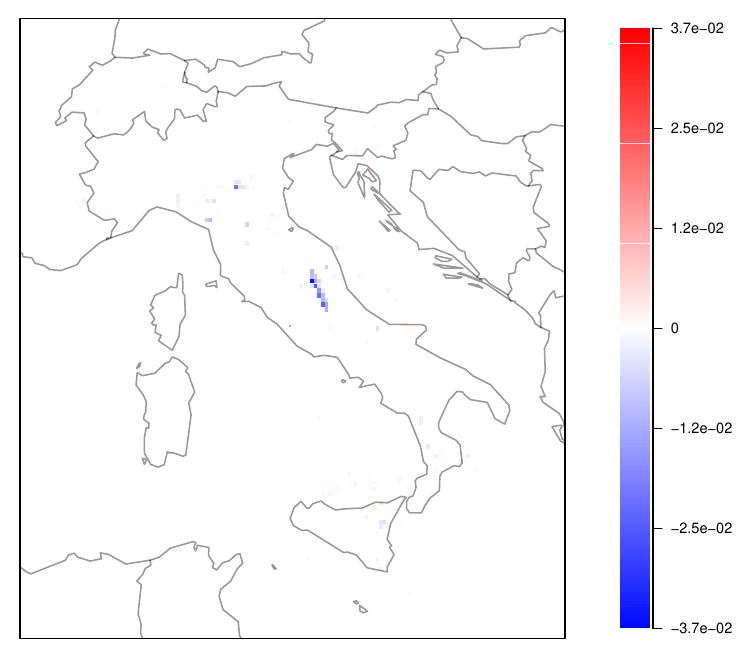}
\end{minipage}
\begin{minipage}{.495\textwidth}
	\includegraphics[clip, width = \textwidth, trim = 8 0 10 0]{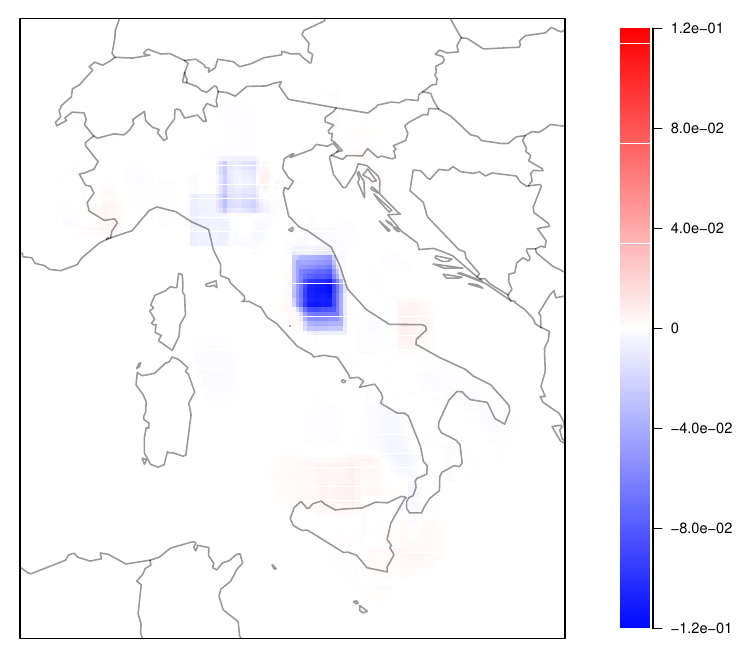}
	\centering
\end{minipage}
\vspace{-2mm}
\caption{Mean score difference based on $S_\mathrm{pois}$ (11)
  between the LM and the LG model, without (left) and with (right) aggregation. Negative values (blue) indicate that the LM model has superior forecast performance, and positive values (red) vice versa.
	\label{fig:spatial_comp_LG}
}
\vspace{5mm}
\begin{minipage}{.495\textwidth}
	\includegraphics[clip, width = \textwidth, trim = 8 0 10 0]{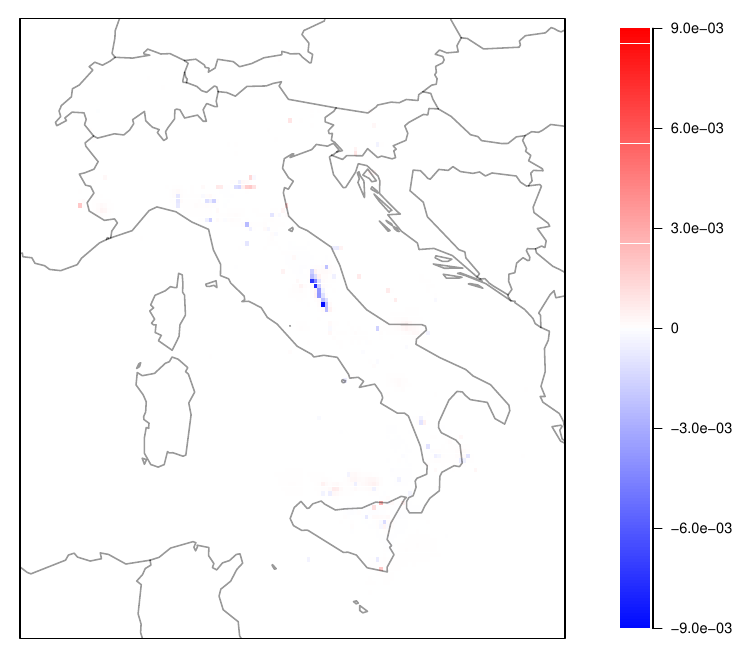}
\end{minipage}
\hfill
\begin{minipage}{.495\textwidth}
	\includegraphics[clip, width = \textwidth, trim = 8 0 10 0]{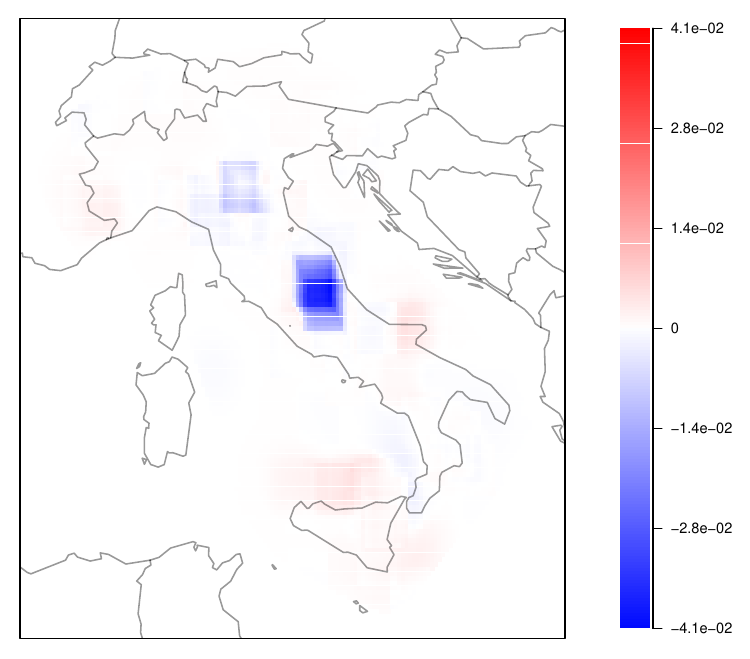}
	\centering
\end{minipage}
\vspace{-2mm}
\caption{Mean score difference based on $S_\mathrm{pois}$ (11)
  between the LM and the SMA model, without (left) and with (right) aggregation. Negative values (blue) indicate that the LM model has superior forecast performance, and positive values (red) vice versa.
	\label{fig:spatial_comp_SMA}
}
\end{figure}

\subsection{Spatial aggregation}
\newsubsection{Spatial aggregation}

We follow the notation of Section~5.2, except that we introduce a
coordinate notation for the testing region (Figure~2).  For each grid
cell $B_i$ we now write $B_{k,l}$ where $k$ is the horizontal and $l$
the vertical coordinate.  A cell with a higher value of $k$ is further
east and a cell with a higher value of $l$ is further north.
Similarly, let $x_{k,l,t}^{(j)}$ be the forecast of model $j$
corresponding to cell $B_{k,l}$ on day $t$.  For combinations of $k$
and $l$ that fall outside the testing region we use the convention
$x_{k,l,t}^{(j)} = 0$ and $B_{k,l} = \emptyset$.

Let $\delta \in \nat_0$ be a given level of aggregation.  We define
the locally aggregated forecast and the locally aggregated grid cell
at coordinate $(k,l)$ and aggregation level $\delta$ via
\begin{align*}
\bar x_{k,l,t}^{(j)} := \sum_{\mu = - \delta}^{\delta} \sum_{\nu = - \delta}^{\delta} x_{k+\mu, l + \nu, t}^{(j)}
\qquad \text{ and } \qquad
\bar B_{k,l} := \bigcup_{\mu = - \delta}^{\delta} \bigcup_{\nu = - \delta}^{\delta} B_{k + \mu, l + \nu}
\end{align*}
respectively.  In the interior of the testing region, this is an
aggregation of the forecasts over a square neighbourhood with edge
length $2\delta + 1$ centred at $(k,l)$.  At the boundary of the
testing region the aggregation neighbourhoods will be smaller, however,
as there are almost no events in this area, this does not affect the
plots.  Due to the linearity of expectations, the values $\bar
x_{k,l,t}^{(j)}$ are again valid mean forecasts that can be compared
via consistent scoring functions, e.g.\ the Poisson score~(9).  The
right plots of Figures~4, \ref{fig:spatial_comp_LG}, 
and~\ref{fig:spatial_comp_SMA} show this comparison via the mean score
difference of the locally aggregated forecasts
\begin{align*}
\bar \Delta_{k,l}^{(j,j')} := \frac{1}{5514} \sum_{t=1}^{5514} \big( S_\mathrm{pois}( \bar x_{k,l,t}^{(j)} , \varphi_t (\bar B_{k,l} ) ) - S_\mathrm{pois}( \bar x_{k,l,t}^{(j')} , \varphi_t (\bar B_{k,l} ) ) \big) ,
\end{align*}
where $\delta = 5$.  For $\delta = 0$ there is no aggregation, so
$\bar \Delta_{k,l}^{(j,j')}$ simplifies to $\Delta_i^{(j,j')}$, the
(non-aggregated) mean score difference~(11).  For $\delta$ large
enough there is essentially only one big grid cell and one forecasted
number remaining.  The corresponding plot would show only one colour,
indicating the forecast performance of the models with respect to the
total number of events in the testing region.

\subsection{Sample size considerations}
\newsubsection{Sample size considerations}

Point process forecasting is often challenged by a lack of data, and
particularly a lack of data to properly test newly proposed prediction
models.  In this light, a critical question is how much data is
required to reach valid conclusions on superior predictive ability.  As
discussed, a commonly used tool is the Diebold--Mariano (DM) test, which
is a one-sample $t$-test applied to the score differentials, with adaptations to time series settings.  Standard
power calculations for $t$-tests apply to independent samples, and a
well known, crude rule of thumb \citep{Lehr1992, vanBelle2008} states
that for a one-sample, two-tailed $t$-test with level 0.05, a sample
size $n = 8 s^2/d^2$ yields an approximate power of 0.80, where $s^2$
is the variance of the score differentials, and $d$ is the difference
to be detected.  Phrased differently, if the variance $s^2$ and the
sample size $n$ are given, a difference $d_n = (8 s^2/n)^{1/2}$ is detectable, subject to the above specifications of the size and the power of the $t$-test. 

In Tables \ref{tab:pois} and \ref{tab:quad} we return to Table~1 in the
main paper, where we compare the predictive performance of the LM, FMC,
LG, and SMA models, respectively.  We show the mean score differential
and its variance, and find the detectable difference $d_{5514}$ at the
given sample size of $n = 5514$ daily forecasts of earthquake activity
over the subsequent seven-day period, for the Poisson score and the
quadratic score, respectively.  Figures~\ref{fig:pois} and
\ref{fig:quad} show the sample autocorrelation function for the score
differentials.  Not surprisingly, there is considerable dependency at
lags up to about seven to nine days ahead, due to the overlap in the
seven-day outlook, though autocorrelations are small to negligible at
higher lags.  As standard power calculations assume independent samples, a more
appropriate quantification of a detectable difference is based on a
sample size of $[5514/7] = 787$.  A further alternative is to use an 
estimate of the effective sample size \citep{ThieZwiers1984}, which 
reduces the regular sample size according to the autocorrelation of
the series, in line with the handling of dependencies in DM tests.

Interestingly, under both the Poisson
and the quadratic score, and for each of the six binary model
comparisons, the actual mean score differential $m$ tends to be nested in
between the (overly) optimistic estimate $d_{5514}$ and the (arguably)
realistic estimate $d_{787}$ for a detectable difference, which indicates that the comparative evaluation might reasonably be considered to be based on sufficient data.  Evidently,
this current analysis is crude and preliminary, using default
specifications from the biostatistical literature for size and power, and we encourage
follow-up studies.

\begin{table}[htb]
\caption{Mean $m$ and variance $s^2$ of the score differential, and detectable difference $d_n$ for sample size $n = 787$ and $n = 5514$ according to the rule of thumb by \citet{Lehr1992}, under the Poisson score and for the models from Table 1 in the main paper.  \label{tab:pois}}
\medskip
\centering
\vspace{2mm}
\footnotesize
\begin{tabular}{l cccccc}
\hline 
Poisson score & LG$-$LM & LG$-$SMA & LG$-$FMC & FMC$-$LM & FMC$-$SMA & SMA$-$LM \\
\hline 
Mean $m$ & 0.307 & 0.285 & 0.221 & 0.086 & 0.064 & 0.022 \\
Variance $s^2$ &11.936 & 6.438 & 4.885 & 2.542 & 0.695 & 0.983 \\
\hline
$d_{5514}$ & 0.132 & 0.097 & 0.084 & 0.061 & 0.032 & 0.038 \\
$d_{787}$ & 0.348 & 0.256 & 0.223 & 0.161 & 0.084 & 0.100 \\
\hline
\end{tabular}

\end{table}

\begin{table}[htb]
\caption{Same as Table \ref{tab:pois}, but under the quadratic score. All entries are to be divided by a factor of 100.  \label{tab:quad}}
\medskip
\centering
\vspace{2mm}
\footnotesize
\begin{tabular}{l cccccc}
\hline 
Quadratic score & LG$-$LM & FMC$-$LM & SMA$-$LM & LG$-$SMA & FMC$-$SMA & LG$-$FMC \\
\hline 
Mean $m$ & 0.563 & 0.505 & 0.293 & 0.270 & 0.211 & 0.058 \\
Variance $s^2$ &1.690 & 1.303 & 0.605 & 0.295 & 0.159 & 0.152 \\
\hline
$d_{5514}$ & 0.495 & 0.435 & 0.296 & 0.207 & 0.152 & 0.149 \\
$d_{787}$ & 1.311 & 1.151 & 0.784 & 0.548 & 0.402 & 0.393 \\
\hline
\end{tabular}

\end{table}

\begin{figure}[p]
\centering
\vspace{3mm}
\includegraphics[width = \textwidth]{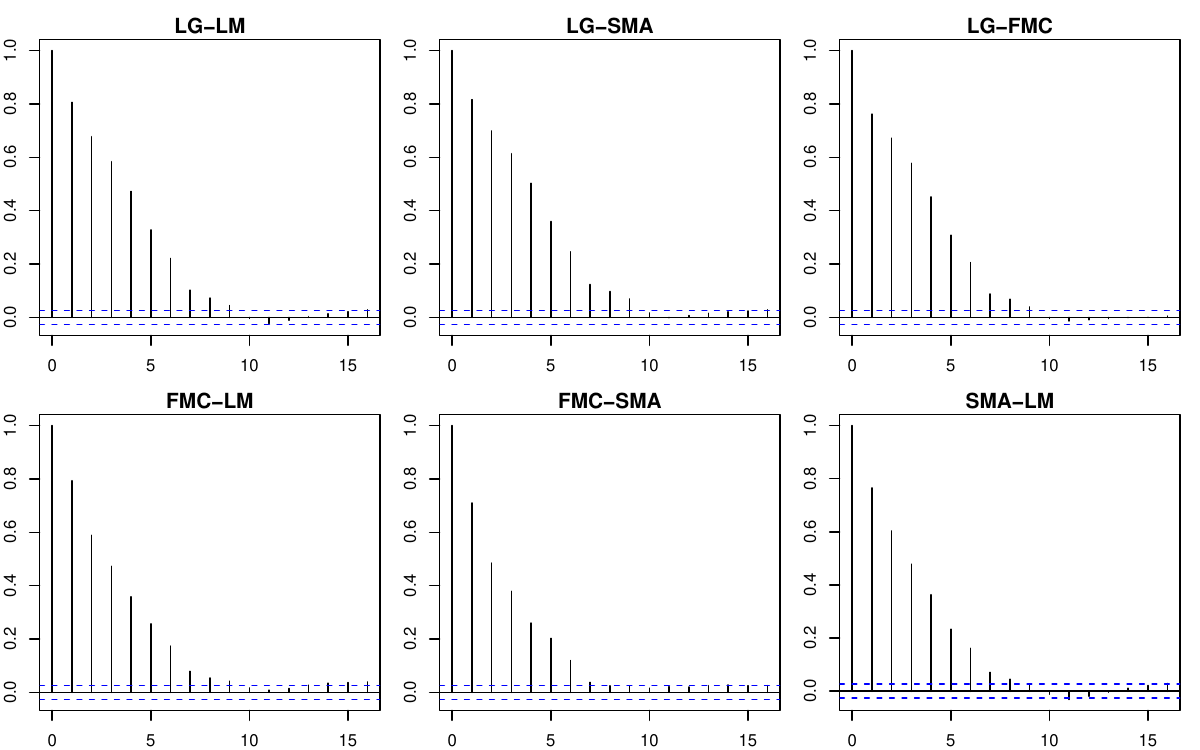}
\caption{Sample autocorrelation function of the Poisson score differentials for the forecasts from Table~1 in the main paper, with lag in days \label{fig:pois}}
\end{figure}

\begin{figure}[p]
\centering
\vspace{3mm}
\includegraphics[width = \textwidth]{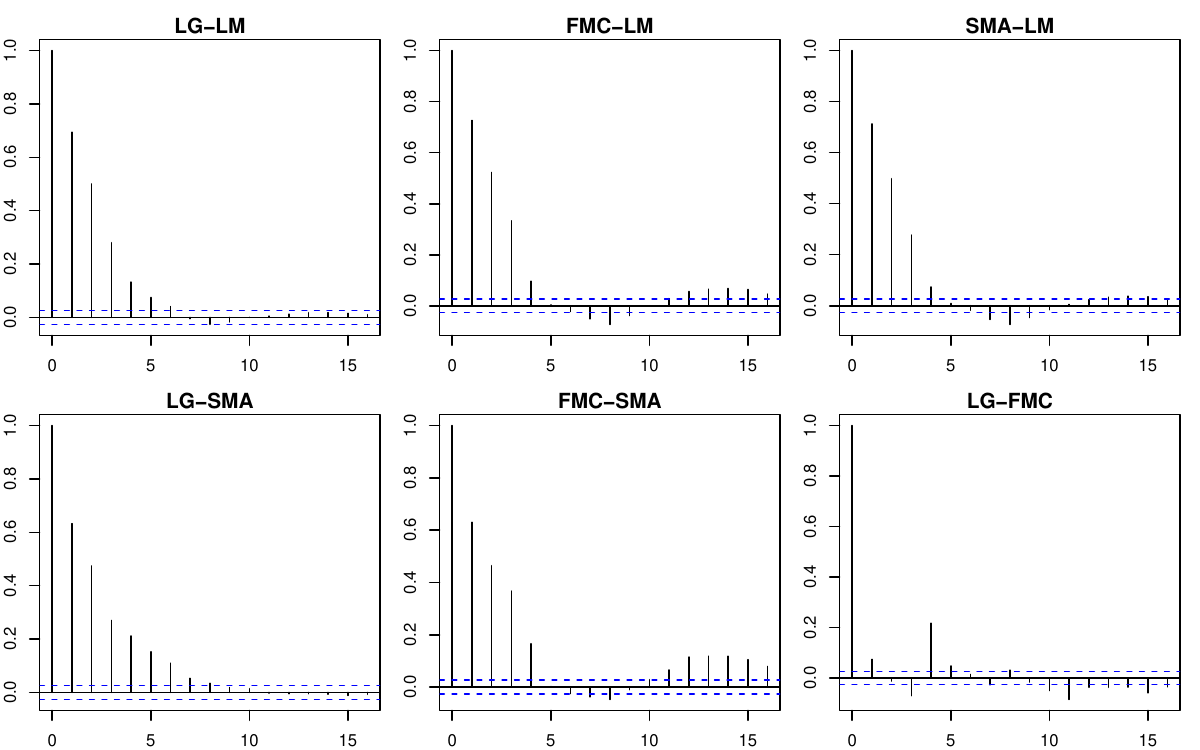}
\caption{Same as Figure \ref{fig:pois}, but under the quadratic score  \label{fig:quad}}
\end{figure}

{\small
\bibsep=0pt
\setlength{\bibsep}{0pt}

\renewcommand\refname{References (Supplement)}

}

\end{document}